\documentclass{amsart}
\usepackage{amssymb, amsmath, color}
\usepackage[latin1]{inputenc}
\usepackage{graphicx}
\usepackage{color}

\def\black{\color{black}}

\newtheorem{theorem}{Theorem}
\newtheorem{lemma}[theorem]{Lemma}
\newtheorem{remark}[theorem]{Remark}
\newtheorem{definition}[theorem]{Definition}
\newtheorem{corollary}[theorem]{Corollary}
\newtheorem{example}[theorem]{Example}

\begin{document}

\title[Relative position between a hyperboloid and a sphere]
{Classification of the relative positions between a hyperboloid and a sphere}
\author[Brozos-V\'{a}zquez \, Pereira-Sáez \, Souto-Salorio \,  Tarrío-Tobar]{M. Brozos-V\'{a}zquez \, M.J. Pereira-Sáez \, M.J. Souto-Salorio \, Ana D. Tarrío-Tobar}
\address{MBV: Departmento de Matem\'{a}ticas, Escola Polit\'ecnica Superior, Universidade da Coru\~na, Spain}
\email{miguel.brozos.vazquez@udc.gal}
\address{MJPS: Departamento de Economía Aplicada II, Facultade de Economía e Empresa,
Universidade da Coru\~na, Spain}
\email{maria.jose.pereira@udc.es}
\address{MJSS: Departamento de Computación, Facultade de Informática,
Universidade da Coru\~na, Spain}
\email{ maria.souto.salorio@udc.es}
\address{ADTT: Departmento de Matem\'{a}ticas, {Escola Técnica de Arquitectura,}
Universidade da Coru\~na, Spain}
\email{madorana@udc.es}
\thanks{Supported by projects  EM2014/009 and MTM2013-41335-P with FEDER funds (Spain).}

\begin{abstract}
We characterize all possible relative positions between a hyperboloid of one sheet and a sphere through  the roots of a characteristic polynomial associated to these quadrics. The classification is also suitable for a hyperboloid and a ellipsoid in some situations.  

As an application, this provides a method to detect contact between the two surfaces by a simple calculation in many real world applications.
\end{abstract}

\maketitle

\section{Introduction}\label{sect:introduction}

Given two arbitrary surfaces, a natural problem is to study how one is placed in relation with the other and, in particular, if the two of them intersect. The latter, i.e. the problem of detecting if two surfaces (or, more generally, two bodies) intersect is referred to in the literature as {\it contact detection}.
This is a broad problem of great interest nowadays. 
Problems of contact detection are found  in many fields such as 
computer graphics, robotics, computational physics,  mechanical 
systems in geomechanics, humanoid design in biomechanics,  animation and computer simulated environments among others (see for example   \cite{9}).


The contact detection problem between two arbitrary bodies is far too complicated due to the lack of geometric representativity. Thus, in most cases, one uses models with a geometric description which makes it easier to handle. Quadrics are geometric entities that issue such a description for a large variety of shapes and have proved to be suitable. On the one hand, quadrics are interesting by themselves in many of the subjects cited above; see, for example, \cite{Choi2003}, \cite{Farouki1989}, \cite{levin76}, \cite{levin78}, \cite{Lingeo}, \cite{ambrosio}, \cite{Wang 2004} and \cite{wank-wang-kim}. Also, conics are of interest in dimension two and have been considered in \cite{Choi2006} and \cite{etayo2006}. On the other hand, quadrics provide a variety of shapes to approximate more complicated surfaces or to be used as bounding volumes. 
Many types of bounding volumes have been proposed and the choice within the bounding volume hierarchy has a direct impact on the collision detection query for volume based methods. Because of its simplicity, spheres are the most widely used (see \cite{39} and \cite{34}). Welzl algorithm to find minimum bounding spheres is presented for example in \cite{44} and \cite{42}.

Generically, the literature about contact detection studies contact between convex bodies.
Moreover, most of the proposed algorithms  are indeed valid only under the condition of convexity. Our proposal is, however, to study a situation where one of the bodies 
is not convex (the hyperboloid).

The characteristic polynomial we study here can be defined for any pair of quadrics and was previously used to analyze  the relative position between conics in \cite{etayo2006} and \cite{liu2004}, or to detect contact of two ellipsoids in \cite{wank-wang-kim}. Furthermore, it can be adapted to the case of two moving ellipsoids \cite{X2011}. As we shall see presently, by the  analysis of these roots we can detect the relative position of the two quadrics and see if the two quadrics intersect, if they are tangent or if they do not touch each other. 

In this paper, we give a classification of the relative positions  between a sphere and a general circular hyperboloid of one sheet, providing the complete picture of all possible relative positions by means of the roots of the characteristic polynomial.

Our approach gives geometric information about some singular positions between the quadrics, which are represented by multiple roots of the characteristic polynomial. 
So, in addition to the techniques of Linear Algebra we come up with a geometric point of view. Those readers who just want to know the relative position between a hyperboloid and a sphere only have to solve the characteristic polynomial and go to Tables I, II or III to get how the quadrics are located.

The remainder of the paper is organized as follows. We start by setting the context, explaining the problem formally and introducing the ingredients which take part in it in Section~\ref{subsect:setting-context}. Then we state the main results of the paper (Section~\ref{sect:main-results}), that summarize the conclusions of the analysis we will carry out later. These are Theorem~\ref{th:th1}, Theorem~\ref{th:th2} and Theorem~\ref{th:th3}. Theorem~\ref{th:th1} characterizes general relative positions between a circular hyperboloid of one sheet and a sphere by means of the roots of the characteristic polynomial. Corollary~\ref{co:rlessa} and Remark~\ref{remark:Cardano} provide a method to be used in real world situations as an application of the theorem and an specific example is presented. Theorem~\ref{th:th2} and Theorem~\ref{th:th3} deal with two special cases which appear only under certain conditions which depend on the geometry of the hyperboloid, the sphere and the relation between them. This completes the global study. Two more examples are given to illustrate how to apply the theorems to identify the relative positions. 

Section~\ref{sect:analysis} is devoted to technical results we will need to prove our main results. We  begin with some remarks on the roots of the characteristic polynomial in Section~\ref{subsect:remarks-characteristic-polynomial}.  Along Section~\ref{subsect:tangency-multiplicity} we study the relations between tangency and the multiplicity of roots. A technique based on moving a sphere along a continuous path is establish in Section~\ref{subsect:moving-sphere} and, finally, the last lemmas in Section~\ref{subsect:characterization-relative-positions} characterize different relative positions between the two quadrics: the non-contact possible situations and the cases in which the intersection curve has two connected components.
We finish the exposition proving in Section~\ref{sect:proofs-ths} the results stated in Section~\ref{sect:main-results}. 

Thus, all the main results and conclusions are concentrated in Section~\ref{sect:main-results}, whereas Sections~\ref{sect:analysis} and \ref{sect:proofs-ths} are devoted to the more technical results, analysis and proofs. 

\subsection{Setting the context.}\label{subsect:setting-context}
We are considering a circular hyperboloid of one sheet. If it is centered at the origin and given in standard form, its equation is
\[
\mathcal{H}:\frac{x^2}{a^2}+\frac{y^2}{a^2}-\frac{z^2}{c^2}=1\,  \text{ for } a,c>0.
\]
We use the standard notation of projective space to associate a $4\times 4$ matrix to $\mathcal{H}$ as follows. Let $X=(x,y,z,1)^t$ be a generic vector in homogeneous coordinates that is not at infinity. Then $X^tHX=0$ is the equation of $\mathcal{H}$
with associated matrix $H$ given by
\[
H=\left(   \begin{array}{cccc}
a^{-2} &0&0&0 \\
0&a^{-2} &0&0 \\
0&0&-c^{-2} &0 \\
0&0&0&-1
\end{array} \right ).
\]
When considering the problem of detecting contact with the hyperboloid we do not assume this is centered at the origin and in standard form, but the study of this case is enough for our purposes as we shall justify presently in Remark~\ref{remark:standar-form-hyperboloid}.

Also, we consider a sphere $\mathcal{S}$ of radious $r>0$ with center at $(x_c,y_c,z_c)$:
\[
\mathcal{S}:(x-x_c)^2+(y-y_c)^2+(z-z_c)^2=r^2\,.
\]
Taking homogeneous coordinates $X=(x,y,z,1)^t$ as before we express the sphere $\mathcal{S}$ as: $X^t S X=0$ with
\[
S=\left(   \begin{array}{rrrr}
1&0&0&-x_c \\
0&1&0&-y_c \\
0&0&1&-z_c \\
-x_c&-y_c&-z_c&-r^2+x_c^2+y_c^2+z_c^2
\end{array} \right ).
\]

We are going to study the relative positions between the hyperboloid and the sphere, being of interest when there is contact between them, this is, when they have at least a common point. Of special significance is the tangent position which corresponds to a point of contact between the surfaces with a common tangent plane. We formalize this in the following definition.

\begin{definition}
We say that $\mathcal{H}$ and $\mathcal{S}$ are in {\it contact} if there exists $X$ such that $X^tHX=X^tSX=0$.

Moreover, we say that $\mathcal{H}$ and $\mathcal{S}$ are {\it tangent} at a point $X$ if  $X^tHX=X^tSX=0$ and $Y^tHX=0$ if and only if $Y^tSX=0$, this is, $\mathcal{H}$ and $\mathcal{S}$ are in contact at $X$ and have the same tangent plane at $X$. 
\end{definition}

The hyperboloid divides the space into two pieces, one which is simply connected and other one which is not. We call the first one {\it interior} of $\mathcal{H}$ and the other one {\it exterior} of $\mathcal{H}$. Thus we say that a point $P$ is {\it interior} to $\mathcal{H}$ if $P^t H P<0$ and {\it exterior} to $\mathcal{H}$ if $P^t H P>0$. By extension, we say that $\mathcal{S}$ is interior (or exterior) to $\mathcal{H}$ if every point in $\mathcal{S}$ is interior (or exterior) to $\mathcal{H}$, respectively. 

\begin{definition}
The following polynomial of degree $4$: 
\[
f(\lambda)=\operatorname{det}(\lambda H+S)
\]
is said to be the {\it characteristic polynomial} of $\mathcal{H}$ and $\mathcal{S}$.
\end{definition}

\begin{remark}\rm
Note that $\operatorname{det}(H)\neq 0$ since $a\neq 0$ and $c\neq 0$. Hence $\operatorname{det}(\lambda H+S)=\operatorname{det}(H)\operatorname{det}(\lambda Id+H^{-1}S)$. This shows that $\operatorname{det}(\lambda H+S)=0$ if and only if $\operatorname{det}(\lambda Id+H^{-1}S)=0$ and the roots of the characteristic polynomial are the eigenvalues of $-H^{-1}S$. 

The study of this characteristic polynomial to analyze the relation between two symmetric bilinear forms was considered previously to solve other geometric or algebraic problems. A good example of this is the result for  simultaneous diagonalization of symmetric bilinear forms given in \cite{MW}, where it was shown that two symmetric bilinear forms with associated matrices $\varphi_1$, $\varphi_2$ can be simultaneously diagonalized if and only if there exists a basis of eigenvectors for $\varphi_1^{-1}\varphi_2$.
\end{remark}

\begin{remark}\rm\label{remark:standar-form-hyperboloid}
The roots of the characteristic polynomial are invariant under affine transformations, since $\operatorname{det}(\lambda H+S)=0$ if and only if $\operatorname{det}(\lambda T^t HT+T^tST)= \operatorname{det}(T)^2\operatorname{det}(\lambda H+S)=0$ for any transformation $T$ with $\operatorname{det}(T)\neq 0$. Thus, although we are considering a general circular hyperboloid,  this can be rotated and translated into one which is centered at the origin and with $OZ$ axis. Therefore we are going to focuss exclusively on a circular hyperboloid given in standar form $\mathcal{H}$, as this restriction does not carry a loss of generality.
\end{remark}

\begin{remark}\label{remark:coordinates}\rm
As a rule, we are considering Cartesian coordinates with the center of the sphere at $(x_c,y_c,z_c)$. Hence the characteristic polynomial for $\mathcal{H}$ and $\mathcal{S}$ takes the form:
\[
f(\lambda)=\frac{\left(a^2+\lambda \right) g(\lambda)}{a^4 c^2}\,,
\]
where $g(\lambda)=-(a^2+\lambda)[(c^2-\lambda)(r^2+\lambda)+z_c^2\lambda]+\lambda(c^2-\lambda)(x_c^2+y_c^2).$

However, sometimes it will be convenient to change to cylindrical coordinates, so that $x_c=\rho_c \cos\theta_c$, $y_c=\rho_c \sin\theta_c$ and $z_c=z_c$. Hence $x_c^2+y_c^2=\rho_c^2$ and $z_c$ remains unchanged. In these coordinates $g$ is written as $g(\lambda)=-(a^2+\lambda)[(c^2-\lambda)(r^2+\lambda)+z_c^2\lambda]+\lambda(c^2-\lambda)\rho_c^2$. 

Note that cylindrical coordinates reflect the circular symmetry of the problem under consideration and this is made explicit in the fact that the coordinate $\theta$ does not appear in the expression of the characteristic polynomial.
Moreover, if we intersect $\mathcal{H}$ and $\mathcal{S}$ with the plane $\theta=constant$ that contains $(x_c,y_c,z_c)$, the problem becomes two-dimensional and one can recover the classification referred in \cite{liu2004} from the results in the next section.
\end{remark}

\section{Main results}\label{sect:main-results}

In this section we describe the global picture with all possible relative positions between $\mathcal{H}$ and $\mathcal{S}$. The following theorems provide the configuration of roots which corresponds to each relative position and vice versa, this is the relative position in terms of the roots of the characteristic polynomial. 

\smallskip

There are some particular cases that appear only under certain circumstances, namely that $a\leq r$ or that $c^2<a r$. These particular cases will be analyzed later. In the following theorem we describe the most general situation, without any further assumption on $r$, $a$ or $c$. For the sake of clarity the picture that represents each relative position between $\mathcal{H}$ and $\mathcal{S}$ is given by the intersection with the vertical plane that contains the OZ axis and the center $(x_c,y_c,z_c)$ of $\mathcal{S}$:

\begin{theorem}\label{th:th1}
Let $\mathcal{H}$ and $\mathcal{S}$ be a circular hyperboloid and a sphere. The following relative positions are in one to one correspondence with the configuration of roots showed in the following table.
\begin{center}
\begin{tabular}{|c|c|p{0.21\linewidth}|p{0.347\linewidth}|}
\hline
\multicolumn{4}{|c|}{General situation}\\
\hline
Type& Picture& Description & Roots ($\lambda_1=-a^2$, $\lambda_2$, $\lambda_3$, $\lambda_4$)\\
\hline
I&\vspace{-0.3cm} \includegraphics[scale=0.1]{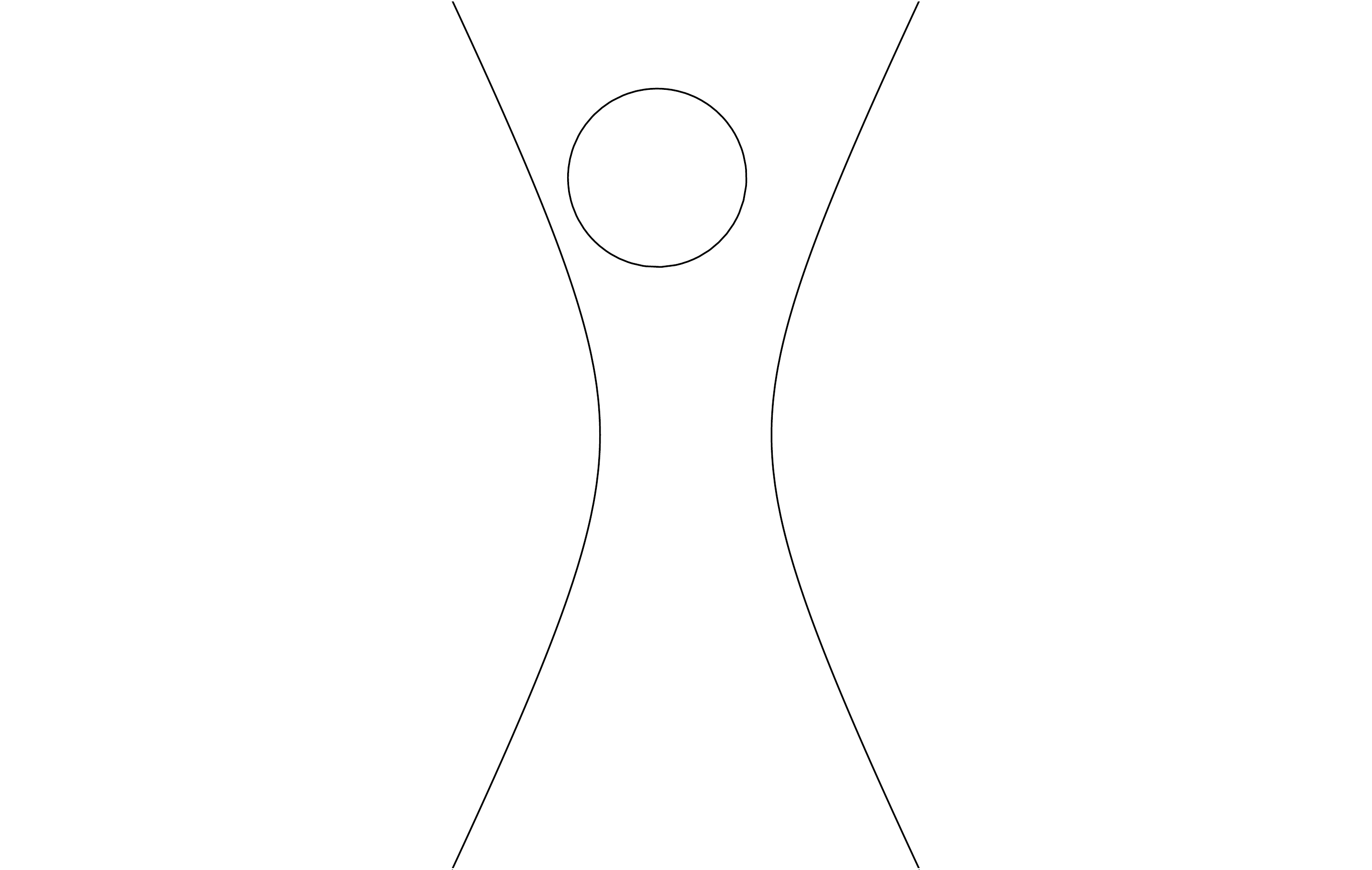}  & $\mathcal{S}$ interior to $\mathcal{H}$ & \vspace{-0.4cm} \begin{center} $-a^2\leq \lambda_2<\lambda_3<0$, $0<\lambda_4$\end{center}\\
\hline
E& \includegraphics[scale=0.1]{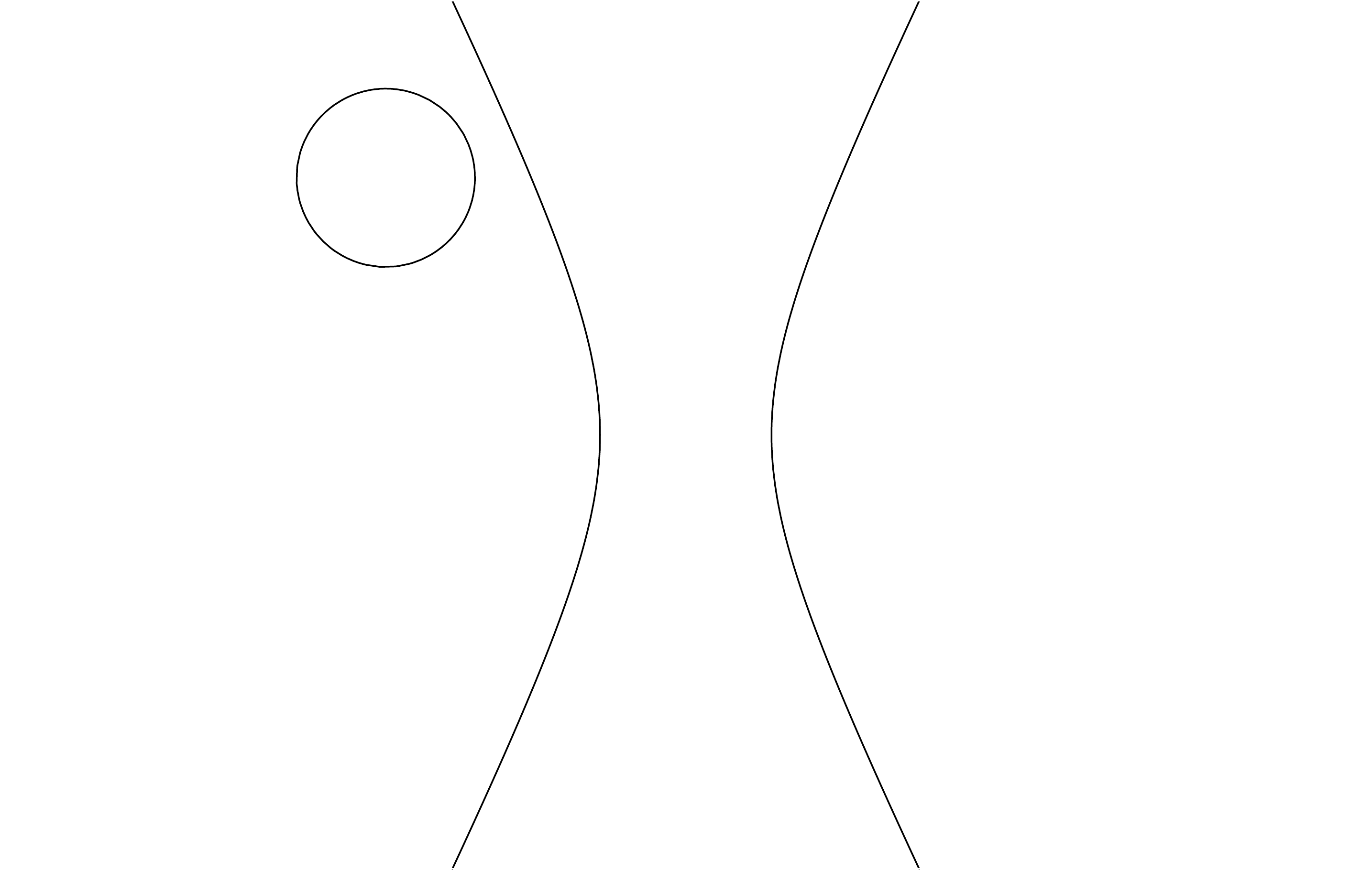}  & $\mathcal{S}$ exterior to $\mathcal{H}$ &
\vspace{-1.1cm}\begin{tabular}{c}
$0<\lambda_2<\lambda_3\leq c^2 <\lambda_4$\\
or\\  $0<\lambda_2<\lambda_3< c^2\leq\lambda_4$
\end{tabular}
\\
\hline
TI& \includegraphics[scale=0.1]{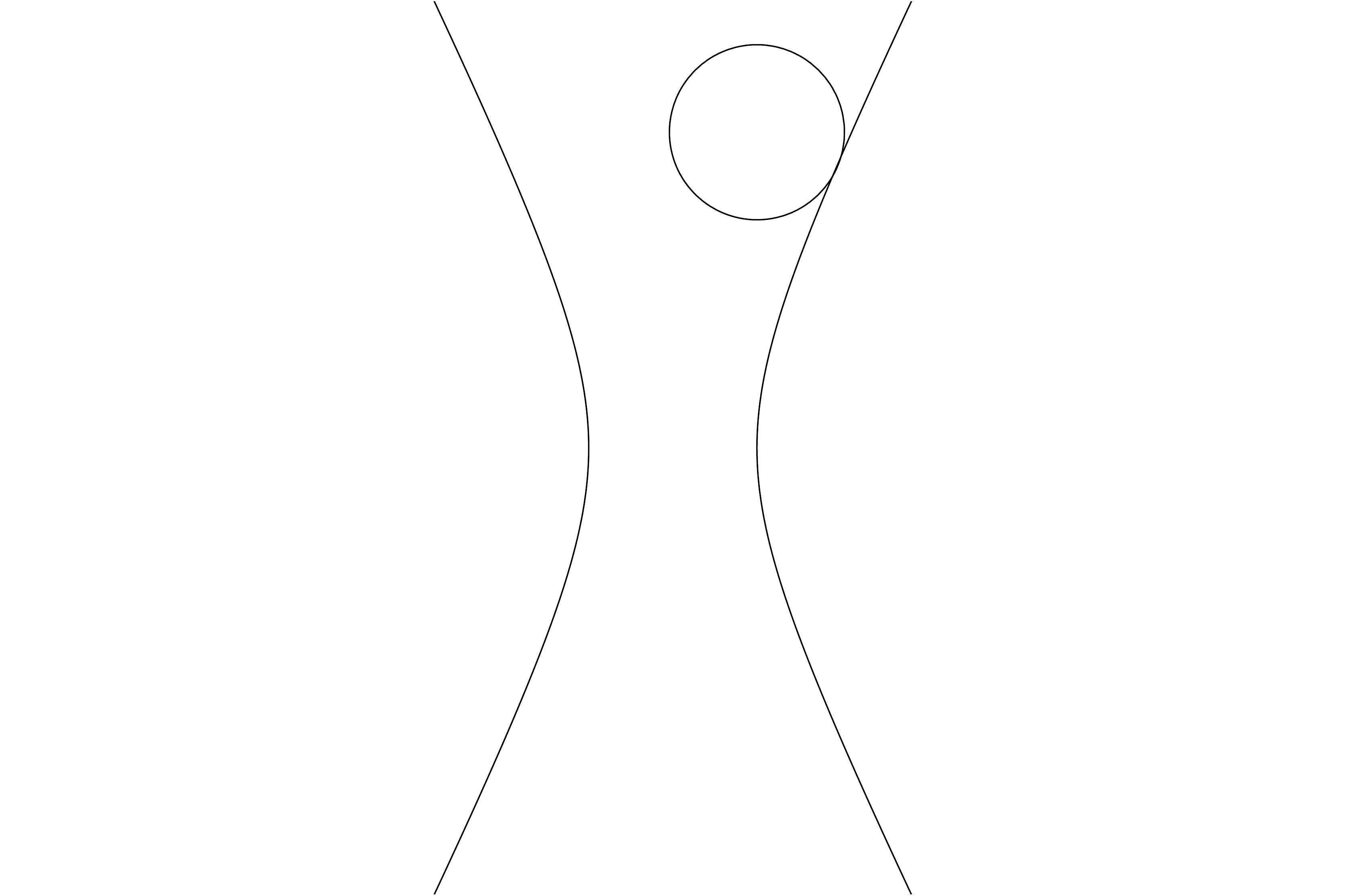}  &  \vspace{-1.1cm}\begin{tabular}{c}
$\mathcal{S}$ tangent  to $\mathcal{H}$\\
\small Center of $\mathcal{S}$\\
\small interior to $\mathcal{H}$
\end{tabular} & \vspace{-0.6cm} \begin{center} $-a^2<\lambda_2=\lambda_3<0<\lambda_4$\end{center}\\
\hline
TE& \includegraphics[scale=0.1]{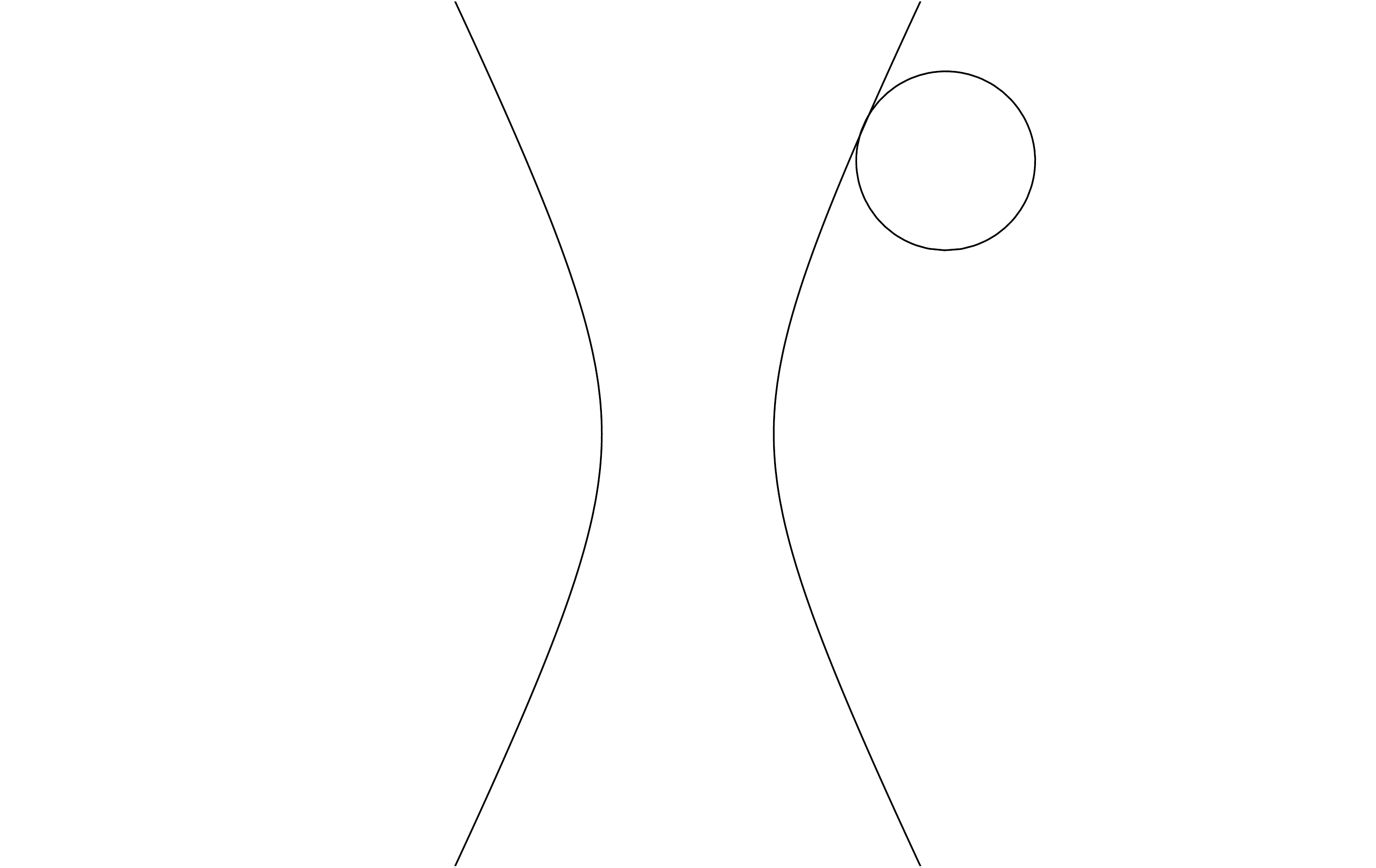}  &\vspace{-1.cm}\begin{tabular}{c} 
$\mathcal{S}$ tangent  to $\mathcal{H}$\\
\small Center of $\mathcal{S}$\\
\small exterior to $\mathcal{H}$
\end{tabular}& \vspace{-0.6cm} \begin{center} $0<\lambda_2=\lambda_3<\lambda_4$\end{center}\\
\hline
C& \includegraphics[scale=0.1]{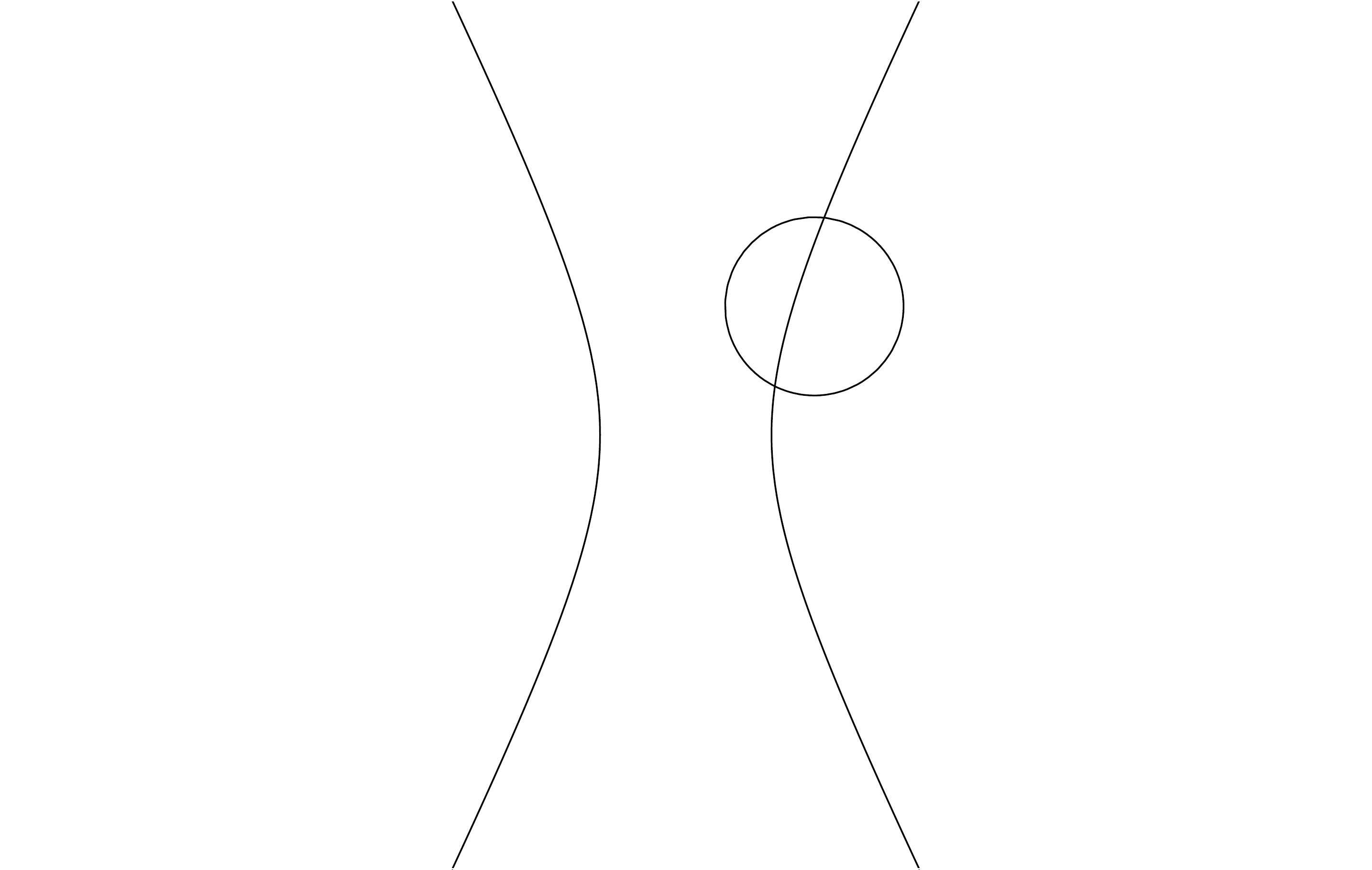}  &  
\vspace{-0.6cm}\begin{tabular}{c}
Non-tangent\\
contact
\end{tabular}
&\vspace{-0.6cm}\begin{center}  $\lambda_2=\bar\lambda_3\in\mathbb{C}$, $0<\lambda_4$\end{center}\\
\hline
\end{tabular}
\flushleft{{\bf Table 1.} General relative positions between $\mathcal{S}$ and $\mathcal{H}$.}
\end{center}
\end{theorem}

\begin{remark}\rm
The table should be interpreted in the following way. The relative position of type $I$ corresponds to the case in which the sphere is interior to the hyperboloid and there is no contact between them, as the associated picture shows. Any hyperboloid and sphere that fits into this description will give rise to a characteristic polynomial whose roots are $\lambda_1=-a^2$, $\lambda_2$, $\lambda_3$ and $\lambda_4$ that satisfy $-a^2\leq \lambda_2<\lambda_3<0$ and $0<\lambda_4$. The converse also holds, this is, if  a hyperboloid and a sphere have a characteristic polynomial with roots $\lambda_1=-a^2$, $\lambda_2$, $\lambda_3$ and $\lambda_4$ satisfying $-a^2\leq \lambda_2<\lambda_3<0$ and $0<\lambda_4$, then the sphere is interior to the hyperboloid (without contact). All other types are interpreted in an analogous way.
\end{remark}

From a practical point of view and in relation with  possible applications of these results, it is likely that we will be considering the radious of the sphere smaller than $a$, so that the sphere can travel along the interior of the hyperboloid without touching the walls. Also, hyperboloids more widely used for practical applications usually satisfy the condition $ar<c^2$, which avoids situations like those described in Table~3 below.  Under these circumstances Theorem~\ref{th:th1} describes all possible relative positions between $\mathcal{H}$ and $\mathcal{S}$, so the global picture gets simpler and can be described as follows.

\begin{corollary}\label{co:rlessa}
Let $\mathcal{H}$ and $\mathcal{S}$ be a circular hyperboloid and a sphere as described above with $r<a$ and $ar<c^2$. Then there is contact between $\mathcal{H}$ and $\mathcal{S}$ if and only if the characteristic polynomial $f$ has complex roots or a double root $\lambda\neq -a^2$.

Moreover, complex roots correspond to non-tangent contact and a double root $\lambda\neq -a^2$ corresponds to tangent contact.
\end{corollary}

\begin{remark}\rm\label{remark:Cardano}
	In the hypothesis of Corollary~\ref{co:rlessa}, contact is detected by the existence of complex or multiple roots of the characteristic polynomial. As we mentioned in Remark~\ref{remark:coordinates}, $f$ can be written as $f(\lambda)=\frac{a^2+\lambda}{a^4c^2}  g(\lambda)$, where $g(\lambda)$ is a polynomial of degree three (see Lemma~\ref{lemma:tecnical-results-1} for details); thus Cardano's formulas can be applied to $g(\lambda)$ to detect contact between $\mathcal{H}$ and $\mathcal{S}$ by a simple calculation as follows. In general, for a monic polynomial $x^3+a_2 x^2+a_1 x+a_0$, the quantities $Q=(3a_1-a_2^2)/9$ and $R=(9a_2a_1-27a_0-2a_1^3)/54$ are defined so that $\Delta=Q^3+R^2$ detects complex and multiple roots (see, for instance,  \cite{Birkhoff}). Thus, a direct analysis of the coefficients of the polynomial $p(\lambda)$ gives the value for $\Delta$ and detects contact between $\mathcal{S}$ and $\mathcal{H}$ as follows:
	\begin{enumerate}
		\item if $\Delta>0$ then there is non-tangent contact,
		\item if $\Delta=0$ then there is tangent contact, and
		\item if $\Delta<0$ then there is no contact.
	\end{enumerate}
\end{remark}
\noindent{\bf Example.} \it We present a specific example of how to apply Theorem~\ref{th:th1}, Corollary~\ref{co:rlessa} and Remark~\ref{remark:Cardano}. Suppose we are given a hyperboloid and a sphere with equations
\[
\mathcal{H}:\frac{x^2}{2.25}+\frac{y^2}{2.25}-\frac{z^2}{2.56}=1\, \text{ and } \mathcal{S}:(x-2.1)^2+(y-2.2)^2+(z-0.3)^2=1.96\,.
\] 
Using the notation of Section~\ref{subsect:setting-context}, we have that  $a=1.5$ and $c=1.6$ specify the hyperboloid and the sphere has radious $r=1.4$ and center at $(2.1,2.2,0.3)$.

In this case the characteristic polynomial has the expression
\[
f(\lambda)=0.0771605 \lambda^4-0.419753\lambda^3-0.0148611\lambda^2+2.09936 \lambda-1.96,
\] 
so $g(\lambda)=\lambda^3-7.69\lambda^2+17.1099\lambda-11.2896$. Since $r<a$ and $ar<c^2$, we apply Corollary~\ref{co:rlessa} and follow Remark~\ref{remark:Cardano} to compute the value $\Delta=-0.340702<0$. Thus we conclude that there is no contact between $\mathcal{H}$ and $\mathcal{S}$. 
If, moreover, we are interested in knowing the specific relative position between the hyperboloid and the sphere, we compute the roots of $f$. We already know that $\lambda_1=-a^2=-2.25$ is a root and we compute the roots of the third order polynomial $g$ to find the other three roots: 
\[
\lambda_2=1.23656,\quad \lambda_3=2.09451 \,\text{ and }\,\lambda_4=4.35893. 
\]
Now we check Table~1 to see that this root configuration corresponds to Type E, this is, $\mathcal{S}$ is exterior to $\mathcal{H}$.\rm 

\medskip
One of the particular situations which deserves special attention is that in which the sphere $\mathcal{S}$ is large in comparison with $\mathcal{H}$, so that it cannot `travel' all along the interior of the hyperboloid without touching it. This happens when $a\leq r$ and the extra relative positions, appart from those already characterized in Theorem~\ref{th:th1}, are described in the following theorem.

\begin{theorem}\label{th:th2}
Let $\mathcal{H}$ and $\mathcal{S}$ be a circular hyperboloid and a sphere such that $a\leq r$. The following relative positions are in one to one correspondence with the configuration of roots showed in the following table.
\begin{center}
\begin{tabular}{|c|c|p{0.24\linewidth}|c|}
\hline
\multicolumn{4}{|c|}{Particular case: $a\leq r$}\\
\hline
Type& Picture& Description & Roots ($\lambda_1=-a^2$, $\lambda_2$, $\lambda_3$, $\lambda_4$)\\
\hline
TIc& \includegraphics[scale=0.1]{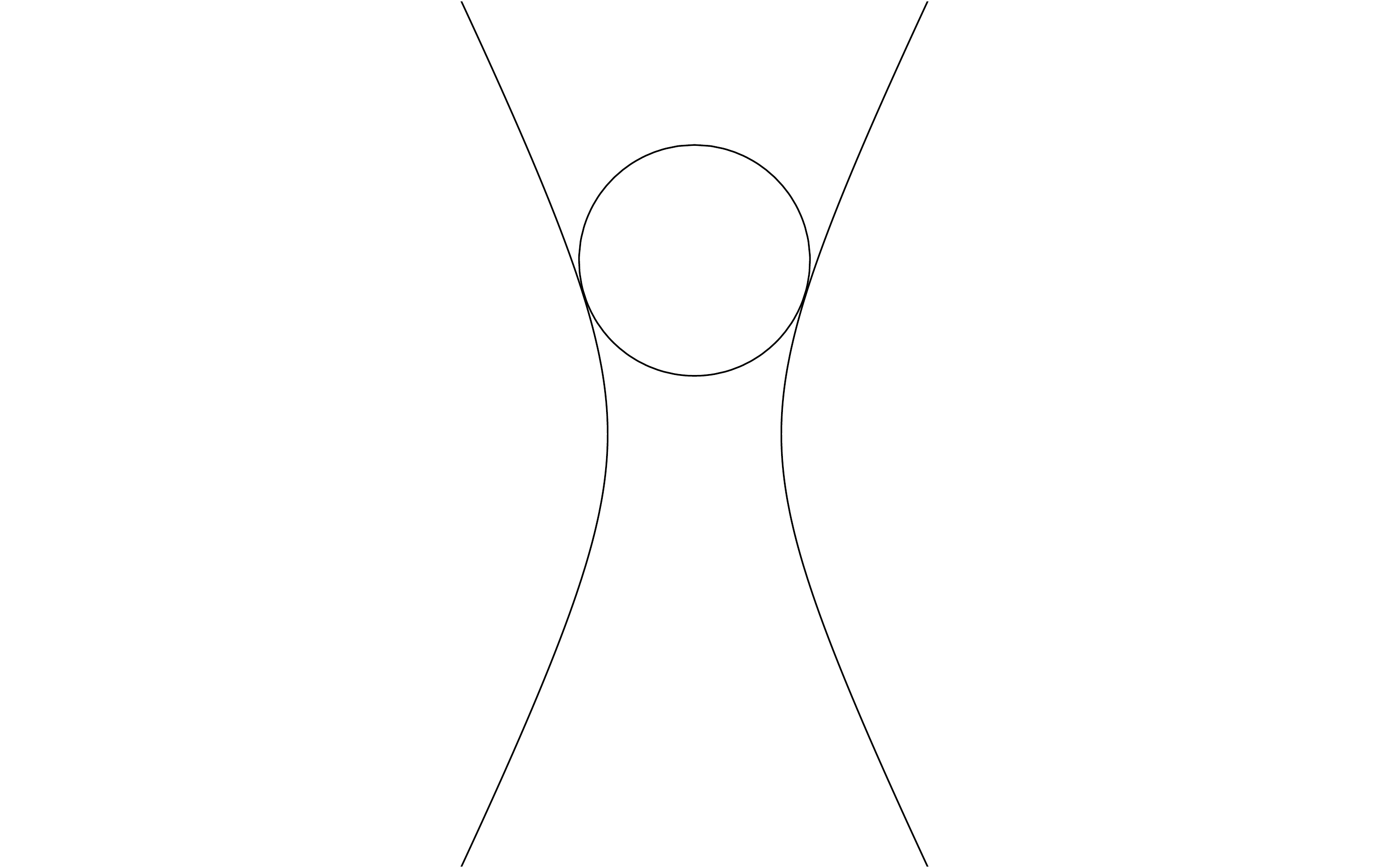}  &\vspace{-0.8cm} Tangency along a circumference
&  $\lambda_2=\lambda_3=-a^2$, $0<\lambda_4$\\
\hline
Td& \includegraphics[scale=0.1]{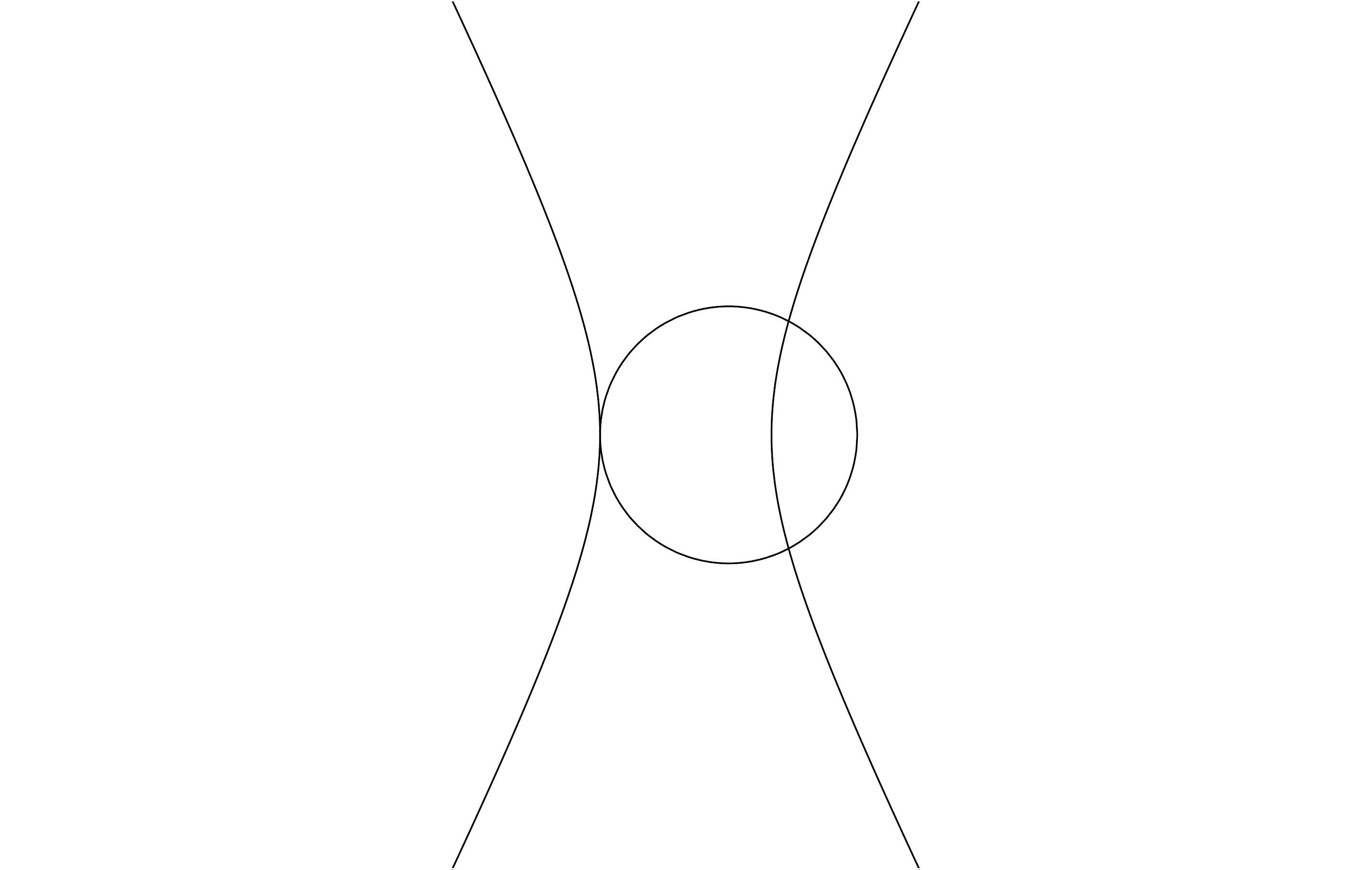}  &\vspace{-0.8cm}  Tangent and non-tangent contact
&  $\lambda_2=\lambda_3<-a^2$, $0<\lambda_4$\\
\hline
Ca& \includegraphics[scale=0.1]{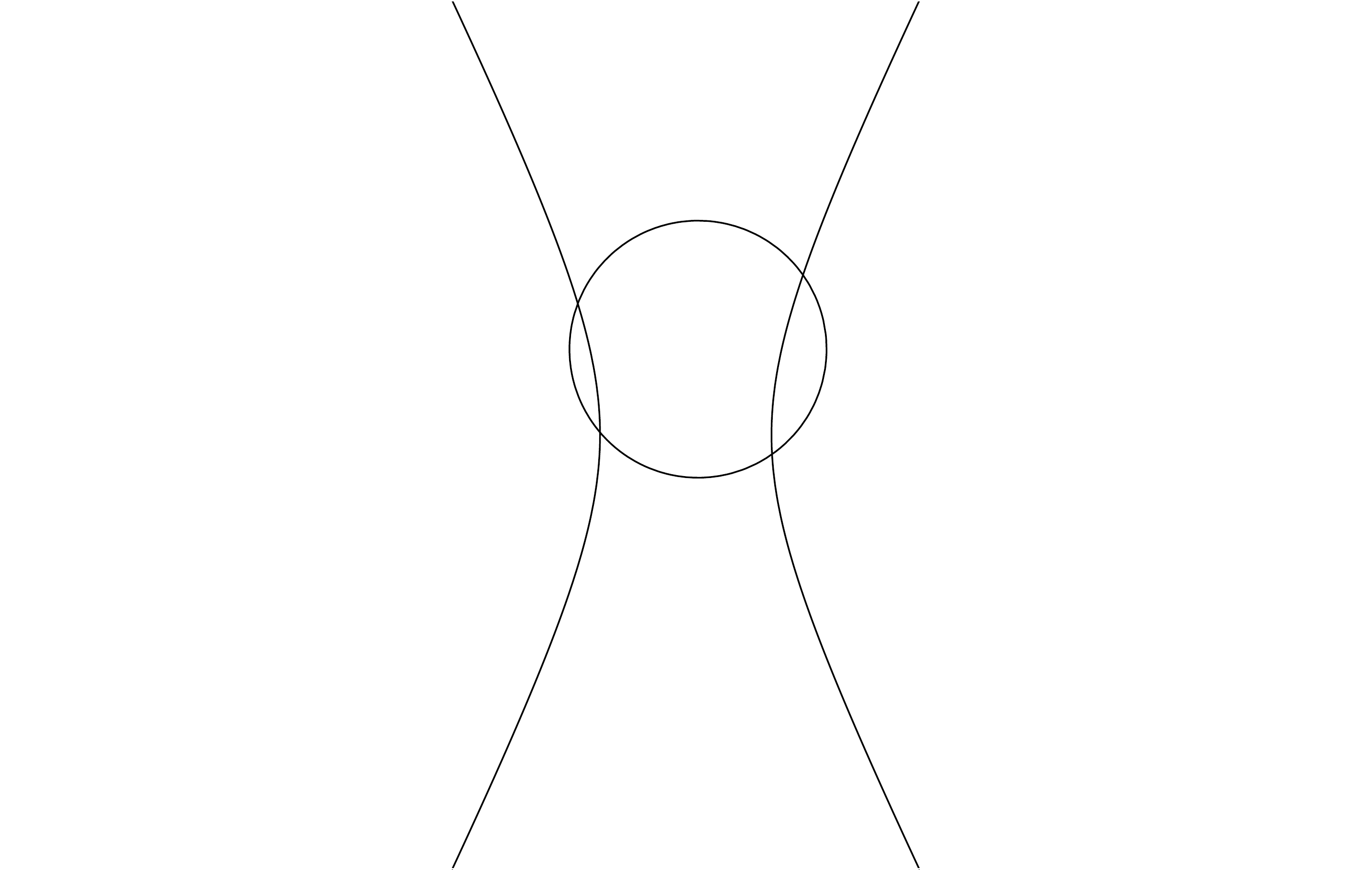}  &  \vspace{-0.8cm} Non-tangent double contact
&  $\lambda_2<\lambda_3\leq-a^2$, $0<\lambda_4$\\
\hline
\end{tabular}
\flushleft{{\bf Table 2.} Extra relative positions between $\mathcal{S}$ and $\mathcal{H}$ if $a\leq r$.}
\end{center}
\end{theorem}

Another particular situation which requires special attention is that in which $c^2< ar$, as this condition allows new relative positions with new configurations of roots for the characteristic polynomial. These are described as follows. 

\begin{theorem}\label{th:th3}
Let $\mathcal{H}$ and $\mathcal{S}$ be a circular hyperboloid and a sphere, respectively,  such that $c^2< ar$. The following relative positions are in one to one correspondence with the configuration of roots showed in the following table. 
\begin{center}
\begin{tabular}{|c|c|p{0.22\linewidth}|c|}
\hline
\multicolumn{4}{|c|}{Particular case: $c^2< ar$}\\
\hline
Type& Picture& Description & Roots ($\lambda_1=-a^2$, $\lambda_2$, $\lambda_3$, $\lambda_4$)\\
\hline
TEs& \includegraphics[scale=0.1]{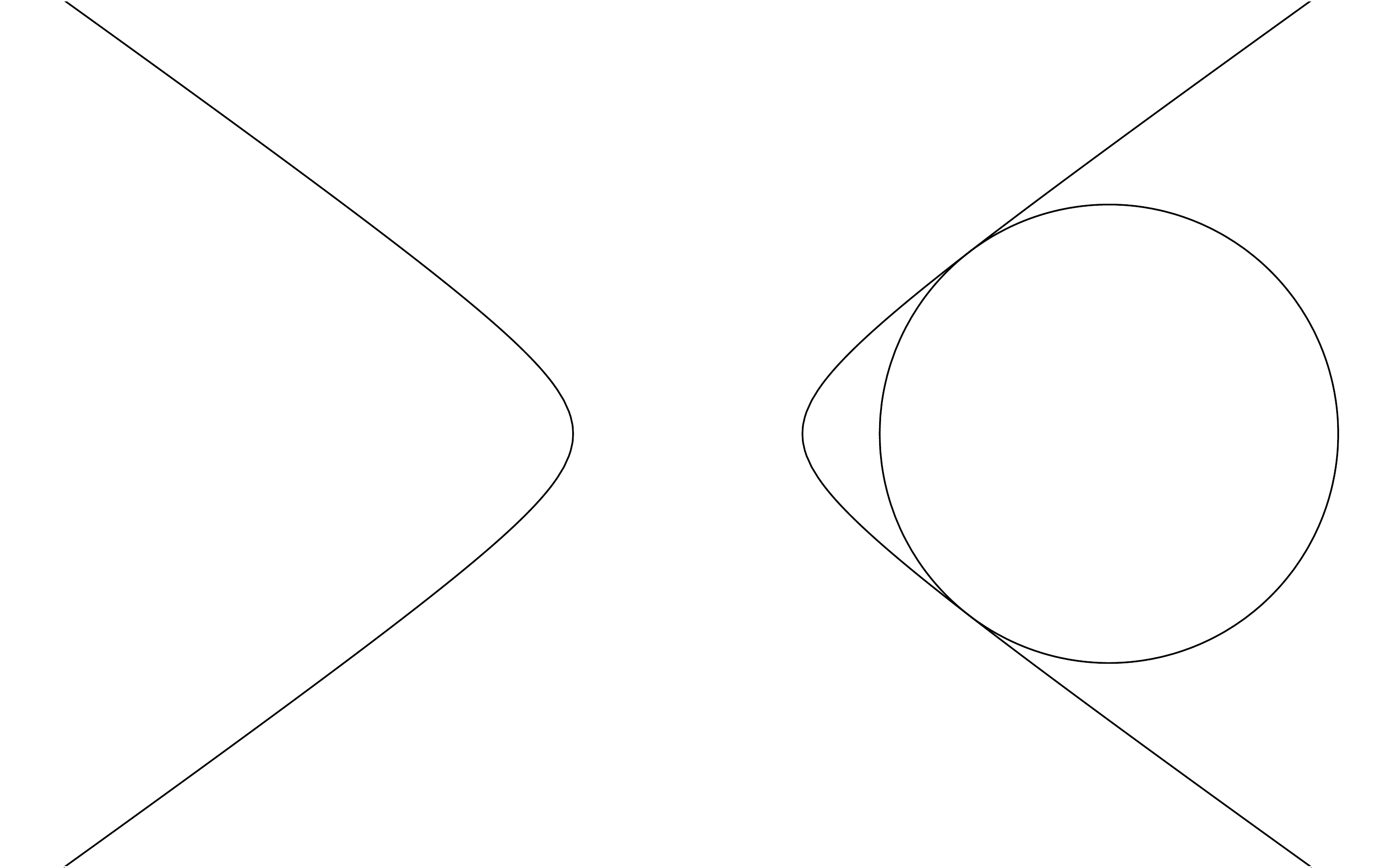}   &\vspace{-0.8cm} Exterior double tangency &  $ 0< \lambda_2=\lambda_3=c^2<\lambda_4$\\
\hline
TEs1& \includegraphics[scale=0.1]{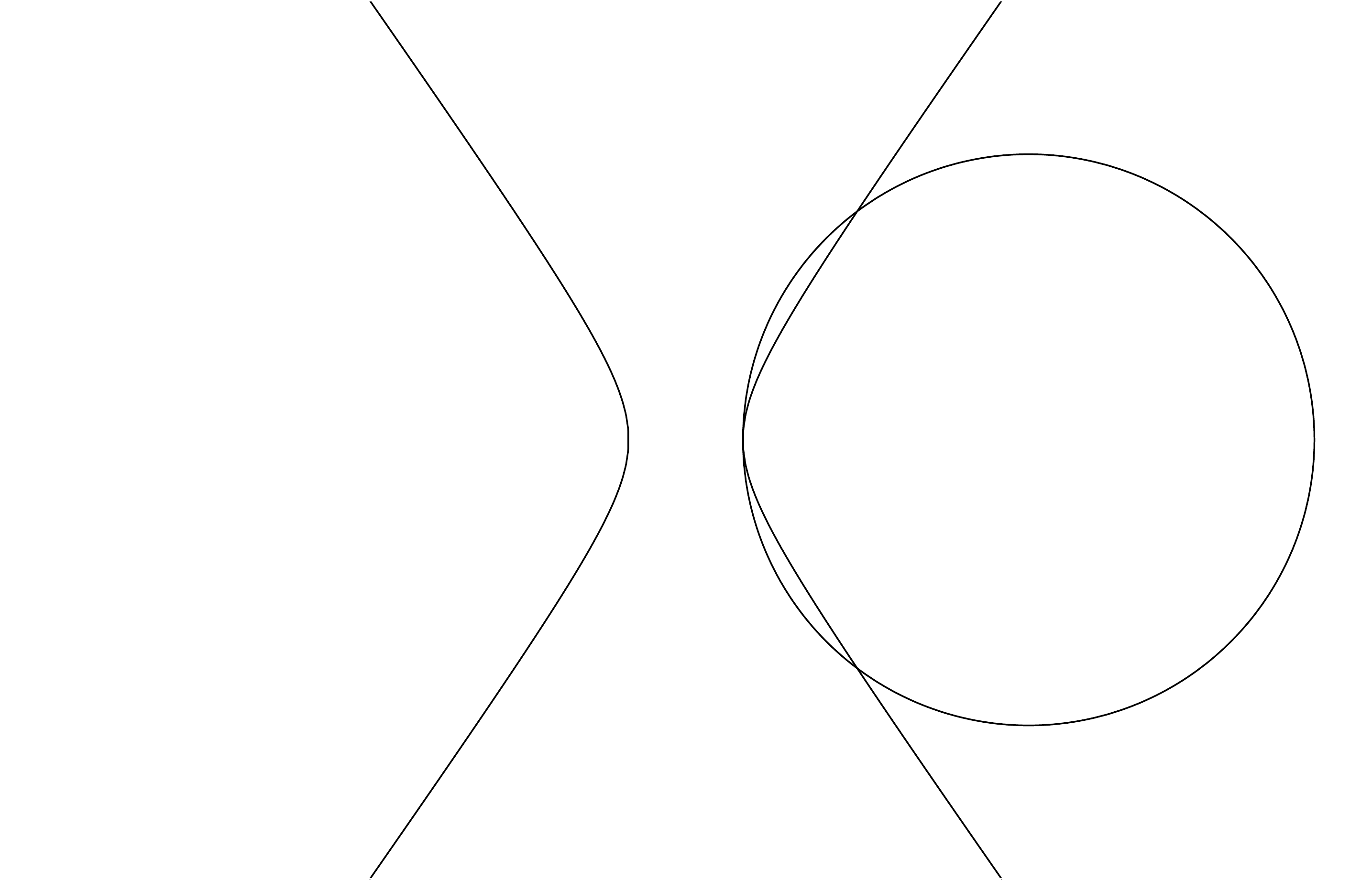}  &\vspace{-0.9cm} Exterior tangency and extra non-tangent contact
& $ 0< \lambda_2=c^2<\lambda_3=\lambda_4=ar$\\
\hline
TEs2&  \includegraphics[scale=0.08]{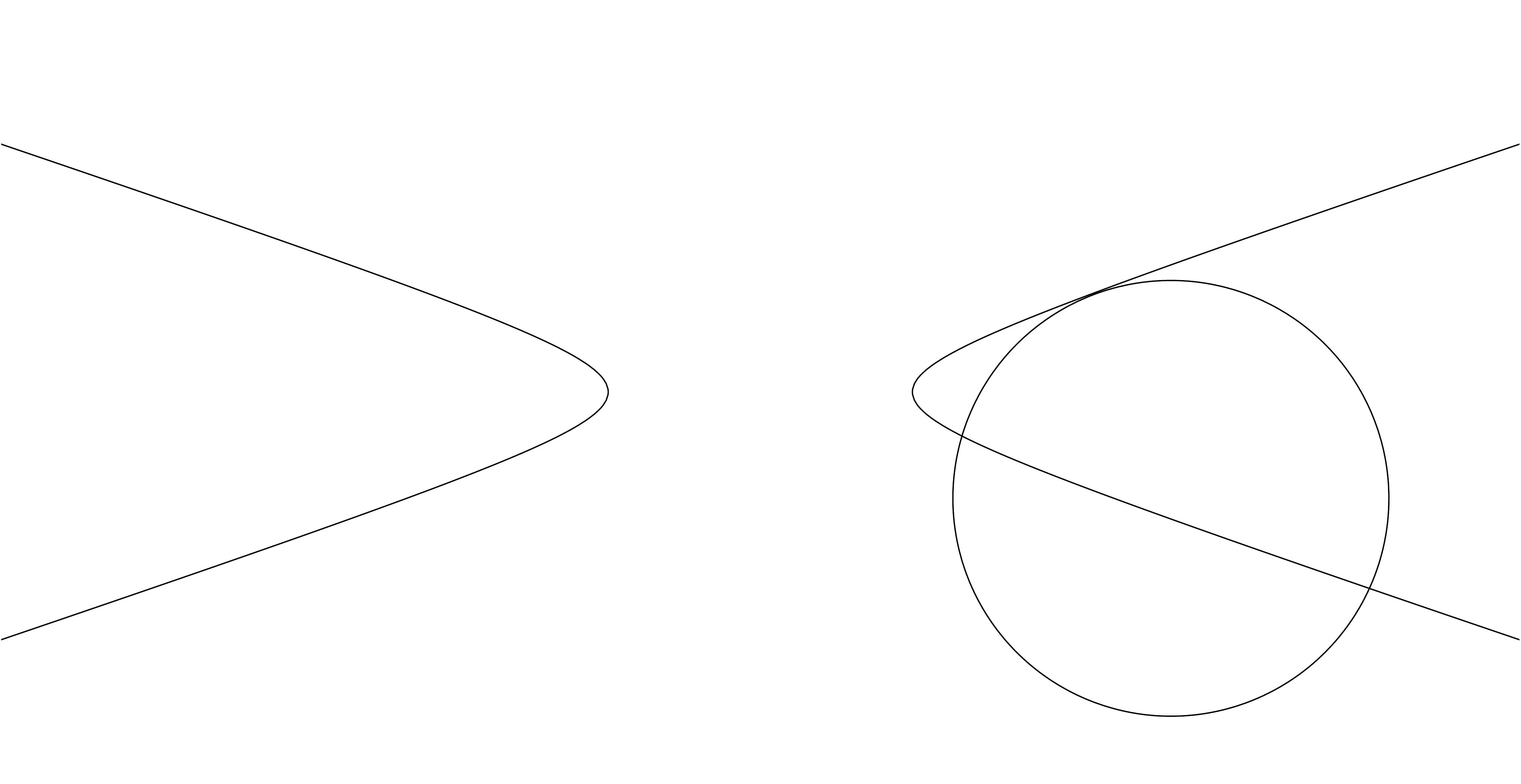} &  \vspace{-0.9cm} Exterior tangency and extra non-tangent contact 
& $0<c^2< \lambda_2=\lambda_3<\lambda_4$\\
\hline
Cm& \includegraphics[scale=0.09]{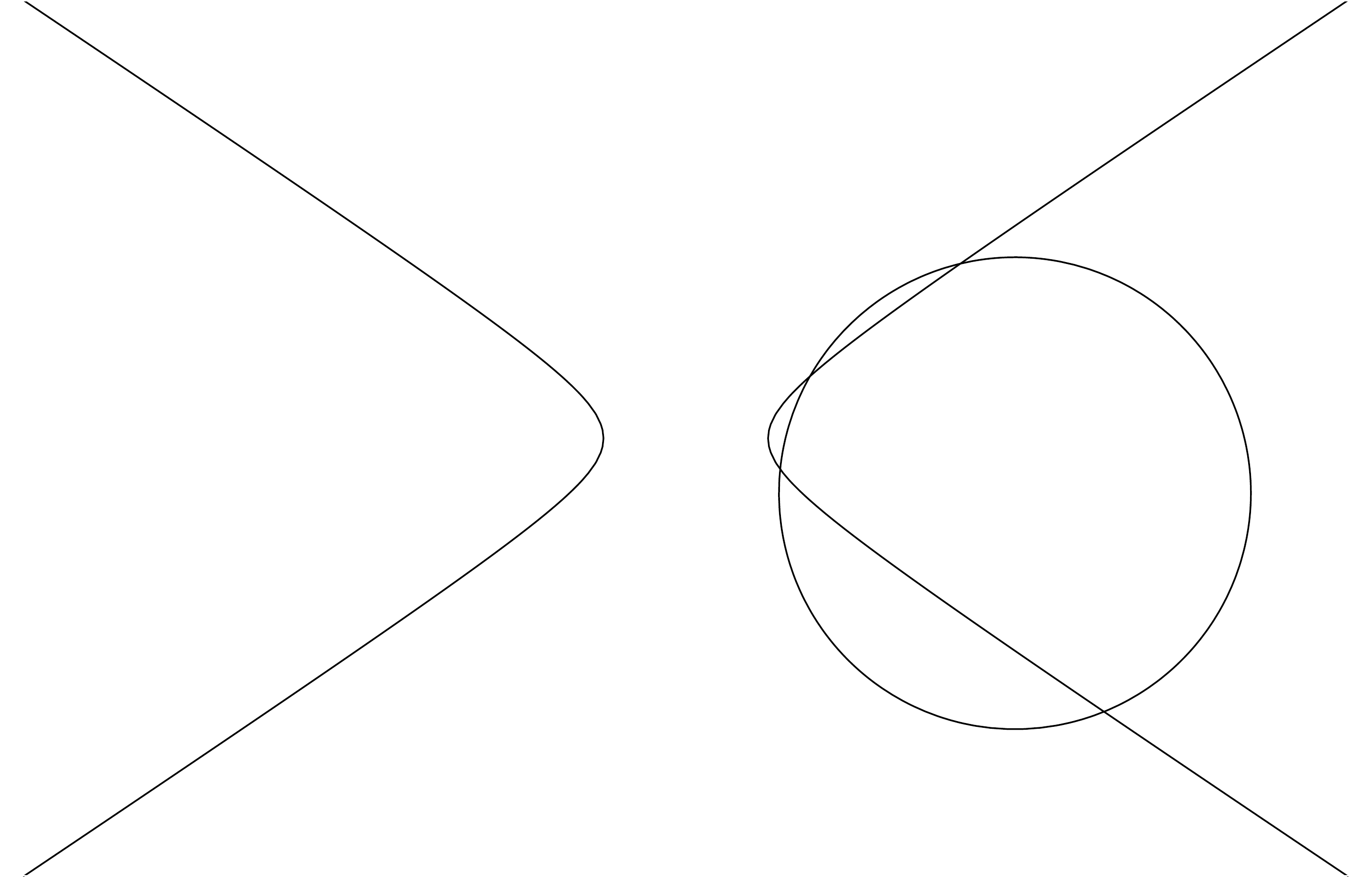}  &
\vspace{-0.8cm} Multiple contact without tangency
& $0<c^2\leq \lambda_2<\lambda_3<\lambda_4$\\
\hline
\end{tabular}
\flushleft{{\bf Table 3.} Extra relative positions between $\mathcal{S}$ and $\mathcal{H}$ if $c^2<a r$.}
\end{center}
\end{theorem}

\begin{remark}\rm\label{remark:relationarc2}
The condition $c^2=ar$ has a very specific geometric meaning. If we intersect $\mathcal{H}$ and $\mathcal{S}$ with a vertical plane containing the $OZ$ axis and the center $(x_c,y_c,z_c)$ of $\mathcal{S}$, we obtain a hyperbola and a circumference, respectively.  Thus, we consider a vertical hyperbola of the hyperboloid (i.e. a hyperbola obtained as the intersection of a plane of the form $\theta=constant$ in cylindrical coordinates) and study its  curvature. For a parametrization of the form $h(t)=(a \cosh(t),c \sinh(t))$ the curvature is given by
\[
\kappa_h(t)=\frac{ac}{\sqrt{(a^2 \sinh^2(t)+c^2\cosh^2(t))^3}},
\]
where we have used that $\kappa_h(t)=\frac{\|h'(t)\times h''(t)\|}{\|h'(t)\|^3}$ (see, for example, \cite{docarmo}). At the intersection point with the plane $z=0$, the curvature of $h$ is maximum and is given by $\kappa_h(0)=\frac{a}{c^2}$. On the other hand, a maximum circumference  of the sphere of radious $r$ has constant curvature $\kappa_c=\frac{1}{r}$. Hence, the equality of both curvatures at the point of intersection with the plane $z=0$ gives rise to the condition
\[
\kappa_h(0)=\frac{a}{c^2}=\frac{1}{r}=\kappa_c\,,
\]
or, equivalently, $c^2=ar$.
Thus, the condition for the existence of situations like those given in Table~3 is precisely $c^2< ar$.

Also note that one can get to the situation where $c^2=ar$ by continuously modifying the values of $a$, $c$ and/or $r$ that satisfy $c^2< ar$. Thus, the situations in Table~3 that represent tangencies, these are $TEs$, $TEs1$ and $TEs2$, reduce to one point of tangency if $c^2=ar$. This is the reason why this equality is not included in Theorem~\ref{th:th3} and, furthermore, we will show in Lemma~\ref{lemma:rootc} that this situation corresponds to a triple positive root $c^2$ for the characteristic polynomial.
\end{remark}

\noindent{\bf Examples.} \it We present here two more examples to show how to apply Theorems~\ref{th:th2} and \ref{th:th3}. Let us take the hyperboloids and the spheres given by equations
\[
\begin{array}{ll}
\mathcal{H}_1:\frac{x^2}{2}+\frac{y^2}{2}-\frac{z^2}{4}=1, & \mathcal{S}_1:x^2+y^2+(z-3)^2=5\,,\\
\noalign{\smallskip}
\mathcal{H}_2:\frac{x^2}{4}+\frac{y^2}{4}-z^2=1, & \mathcal{S}_2:(x-3)^2+(y-3)^2+(z+1)^2=6\,.
\end{array}
\] 
We first consider $\mathcal{H}_1$ and $\mathcal{S}_1$. A straightforward calculation shows that $f(\lambda)= (2 +\lambda)^2 (-20 - 8 \lambda + \lambda^2)/16$ and the roots are:
\[
\lambda_1=\lambda_2=\lambda_3=-2 \,\text{ and }\,\lambda_4=10.
\] 
We check Table~2 to see that this root configuration corresponds to Type TIc, so $\mathcal{H}_1$ and $\mathcal{S}_1$ are tangent along a circumference.

We consider now  $\mathcal{H}_2$ and $\mathcal{S}_2$ and compute $f(\lambda)=(-2 +\lambda)^2 (-24 - 2 \lambda + \lambda^2)/16$. The roots of the characteristic polynomial are:
\[
\lambda_1=-4, \lambda_2=\lambda_3=2 \,\text{ and }\,\lambda_4=6.
\] 
We check Table~3 to see that this root configuration corresponds to Type TEs2, so  $\mathcal{H}_1$ and $\mathcal{S}_1$ are tangent at a point with extra non-tangent contact.\rm

\section{Technical results}\label{sect:analysis}

\subsection{Some remarks on the roots of the characteristic polynomial}\label{subsect:remarks-characteristic-polynomial}

\begin{lemma}\label{lemma:tecnical-results-1}
 The roots of the characteristic polynomial $f$ satisfy:
\begin{enumerate}
\item $-a^2$ is a root.
\item $0$ is not a root.
\item The product of all roots is $-a^4c^2r^2<0$.
\item There exists at least one positive real root $\lambda_4>0$.
\end{enumerate}
\end{lemma}
\proof
By a direct computation one sees that the characteristic polynomial takes the form
\[
f(\lambda)=\frac{\left(a^2+\lambda \right) g(\lambda)}{a^4 c^2}\,,
\]
where 
\begin{equation}\label{eq:poly3}
\begin{array}{rcl}
g(\lambda)&=&-(a^2+\lambda)[(c^2-\lambda)(r^2+\lambda)+z_c^2\lambda]+\lambda(c^2-\lambda)(x_c^2+y_c^2)\\
&=&\lambda^3+(a^2 - c^2 + r^2 - x_c^2 - y_c^2 - z_c^2)\lambda^2\\
\noalign{\smallskip}
&&-(a^2 c^2 - a^2 r^2 + c^2 r^2 - c^2 x_c^2 - c^2 y_c^2 + a^2 z_c^2)\lambda-a^2 c^2 r^2.
\end{array}
\end{equation}
Thus, $-a^2$ is a root of the characteristic polynomial.
Since $f(0)=-r^2$, $0$ is not a root and, as a consequence of Cardano-Vieta formulas, the product of the roots is $-a^4c^2r^2$. Moreover, since $-a^4c^2r^2<0$, not all the roots have the same sign, so there is at least one positive root.\qed

As a consequence of Lemma~\ref{lemma:tecnical-results-1}, a multiple root can be double or triple, but not cuadruple. There are two particular possible roots of $f$ that require an specific analysis, these are $-a^2$ and $c^2$. In the following lemmas we analyze them separately. First we show that $-a^2$ is a multiple root precisely when $\mathcal{S}$ is centered at the OZ axes:

\begin{lemma}\label{lemma:roota}
$-a^2$ is a multiple root of $f$ if and only if $x_c=y_c=0$. 
Moreover, if $x_c=y_c=0$,  $-a^2$ is a triple root if and only if $r^2=a^2+ a^2 z_c^2/(a^2 + c^2)$.
\end{lemma} 
\proof
Suppose $-a^2$ is a multiple root, then $-a^2$ is a root of $g$ in equation~\eqref{eq:poly3}, so
\[
g(-a^2)=-a^2 (a^2 + c^2) (x_c^2 + y_c^2)=0\,.
\]
This shows $-a^2$ is a root of $g$ if and only if $x_c=y_c=0$. If $x_c=y_c=0$ then $g$ reduces to
\[
g(\lambda)=-(a^2 + \lambda) (c^2 r^2 + c^2 \lambda - r^2 \lambda + z_c^2 \lambda - \lambda^2).
\] 
Now, we check that $-a^2$ is a root of $h(\lambda)=c^2 r^2 + c^2 \lambda - r^2 \lambda + z_c^2 \lambda - \lambda^2$ if and only if $r^2=a^2+a^2  z_c^2/(a^2 + c^2)$.
\qed

The particular case in which $\mathcal{S}$ has center at the $OZ$ axis, i.e. $x_c=y_c=0$, will play a distinguished role in our analysis. We summarize the results obtained with $x_c=y_c=0$ in the following corollary.

\begin{corollary}\label{coro:rootatriple}
  If $x_c=y_c=0$ then
\begin{enumerate} 
\item $f(\lambda)$ has four real roots: $-a^2$ which has multiplicity at least $2$ and  $$\lambda_{\pm}=\left(c^2-r^2+z_c^2\pm\sqrt{\left(c^2-r^2+z_c^2\right)^2+4 c^2 r^2}\right)/2.$$
\item  $-a^2$ is a triple root if and only if $\mathcal{S}$ is centered at $(0,0,z_c)$ with 
$$z_c^2= (r^2 -a^2 )(a^2 + c^2)/a^2.$$
\end{enumerate}
\end{corollary} 
\proof
 By Lemma~\ref{lemma:roota}, $f(\lambda)=-h(\lambda)(a^2+\lambda )^2/(a^4c^2)$ where
 $h(\lambda)=c^2 r^2 + c^2 \lambda - r^2 \lambda + z_c^2 \lambda - \lambda^2$. We find the roots of $h(\lambda)$ to see they are those  
given in assertion (1). Since the radicand is positive, $\lambda_{\pm}$ are real. 
Assertion (2) rephrases what was stated in Lemma~\ref{lemma:roota}.
\qed

Now, we turn our attention to the $c^2$ root and show that this is indeed a root if and only if the center of $\mathcal{S}$ belongs to the $XY$-plane:

\begin{lemma}\label{lemma:rootc}
$c^2$ is a root of $f$ if and only if $z_c=0$. Moreover, if $z_c=0$, $c^2$ is a multiple root if and only if $r^2=-c^2+c^2( x_c^2 + y_c^2)/(a^2 + c^2)$. More specifically, 
\begin{enumerate}
\item $c^2$ is a double root if $c^2\neq ar$, and
\item $c^2$ is a triple root if $c^2= ar$. 
\end{enumerate}
\end{lemma}
\proof
We compute $g(c^2)=-c^2 (a^2 + c^2) z_c^2$ to see that $c^2$ is a root of $f$ if and only if $z_c=0$. If $z_c=0$, then 
\[
g(\lambda)=-(c^2 - \lambda) (a^2 r^2 +(a^2 + r^2 - 
   x_c^2 - y_c^2 )\lambda + \lambda^2).
\]
Now, substitute $\lambda$ by $c^2$ in the previous second order factor  to see that $c^2$ has multiplicity at least $2$ if and only if $r^2=-c^2+c^2( x_c^2 + y_c^2)/(a^2 + c^2)$. Assume $c^2$ is a multiple root, then
\[
g(\lambda)=(c^2 - \lambda)^2 (a^2-a^2\frac{ x_c^2+ y_c^2 }{a^2 + c^2}+\lambda),
\]
from where we get the condition $(c^2+a^2)^2=a^2(x_c^2+y_c^2)$ for $c^2$ to be a triple root. Now, from the two conditions one gets the relation $c^4=a^2r^2$, which is equivalent to $c^2=ar$ since $a$, $c$ and $r$ are all positive.
\qed

\begin{remark}\rm \label{remark:rya}Observe from Lemma~\ref{lemma:roota} that if $-a^2$ is a triple root  then $a\leq r$, since $r^2= a^2 + \ast$ where $\ast\in [0,\infty)$. Moreover, if $-a^2$ is a triple root  and $r=a$ then $\mathcal{S}$ is centered at $(0,0,0)$ and intersects $\mathcal{H}$ in a circumference. Furthermore, in this case the positive root is $c^2$, as seen in Lemma~\ref{lemma:rootc}.

Also, recall from Remark~\ref{remark:relationarc2} that the condition given in Lemma~\ref{lemma:rootc} for $c^2$ to be a triple root is precisely that the vertical hyperbola of $\mathcal{H}$ has the same curvature at the point $z=0$ as a maximum circumference of $\mathcal{S}$. 
\end{remark}



\begin{lemma}\label{lemma:distribucion-raices}\mbox{}
\begin{enumerate}
\item If $f$ has three negative roots $-a^2$, $\lambda_2$ and $\lambda_3$. Then either $\lambda_2\leq \lambda_3\leq -a^2$ or $-a^2 \leq\lambda_2\leq \lambda_3< 0$.
\item If $f$ has three positive roots $\lambda_2$, $\lambda_3$ and $\lambda_4$. Then either $\lambda_2\leq \lambda_3\leq c^2 \leq \lambda_4$ or $c^2 \leq\lambda_2\leq \lambda_3\leq \lambda_4$.
\end{enumerate}
\end{lemma}
\begin{proof}
From the proofs of Lemmas~\ref{lemma:roota} and \ref{lemma:rootc} we know that $g(-a^2)=-a^2 (a^2 + c^2) (x_c^2 + y_c^2)$
and that $g(c^2)=-c^2 (a^2 + c^2) z_c^2$, so $g(-a^2)\leq 0$ and $g(c^2)\leq 0$. Since $\lim_{\lambda\to \pm\infty} g(\lambda)=\pm \infty$, we have that the relations $\lambda_2<-a^2<\lambda_3$ and $\lambda_2<c^2<\lambda_3$ for the roots of $g$ are not possible. Hence the lemma follows.
\end{proof}

\subsection{Tangency points and multiple roots}\label{subsect:tangency-multiplicity}

A key point in our global analysis is the relation between tangency and the multiplicity of roots. This is the main subject of the current section.

We consider the roots $-a^2$ and $c^2$ separately as previously. We begin by analyzing the root $-a^2$ as a triple root and the corresponding tangency as follows.
\begin{lemma}\label{lemma:rootatangent}
$-a^2$ is a triple root of $f(\lambda)$ if and only if $\mathcal{H}$ and $\mathcal{S}$ are tangent along a circumference.
\end{lemma}
\proof
By Lemma~\ref{lemma:roota}, a necessary condition for $-a^2$  to be a triple root is that $x_c=y_c=0$. Note also that this is a necessary condition for tangency along a circumference. Thus, assume $x_c=y_c=0$ henceforth.
We study the solutions of the system given by the equations $X_0^t HX_0=0$ and $X_0^tSX_0=0$, which for $X_0^t=(x,y,z,1)$ in homogeneous coordinates are:
\[
\frac{x^2}{a^2} + \frac{y^2}{a^2} - \frac{z^2}{c^2}=1 \text{ and } x^2+y^2+(z-z_c)^2=r^2\,.
\]
Substitute $x^2+y^2$ in the previous equations to see that
\[
\frac{r^2-(z-z_c)^2}{a^2} - \frac{z^2}{c^2}=1\,.
\]
Due to the symmetry of the problem, a tangent solution corresponds to a unique solution of this equation for $z$. Thus, the discriminant $-c^2 (a^4 - c^2 r^2 + a^2 (c^2 - r^2 + z_c^2))$ must vanish. Since $c\neq 0$, this is equivalent to $r^2= a^2 (1 + z_c^2/(a^2 + c^2))$. But this is precisely the condition for $-a^2$ to be a triple root, which was given in Lemma~\ref{lemma:roota}. 
\qed

\medskip

As we saw in Remark~\ref{remark:rya}, a necessary condition for  $-a^2$ to be a triple root is that $a\leq r$. Furthermore, it is also a necessary condition for the sphere to be tangent to the hyperboloid on a circumference and both facts are related by Lemma~\ref{lemma:rootatangent}. Now we are going to analyze the possible relative positions when the center of the sphere is at the $OZ$ axis.

\begin{lemma}\rm\label{lemma:eixoz} 
Let $x_c=y_c=0$. If $r<a$ then the roots $\lambda_i$, $i=1,\dots,4$, satisfy  $-a^2=\lambda_1=\lambda_2< \lambda_3<0<\lambda_4$,
whereas if $a\leq r$ then one of the following possibilities holds:
\begin{enumerate}
\item $\mathcal{S}$ and $\mathcal{H}$ are in non-tangent contact (see Picture Ca in Table 2) and  the roots $\lambda_i$, $i=1,\dots,4$, satisfy 
\[\lambda_2<-a^2=\lambda_1=\lambda_3<0<\lambda_4.\]
\item $\mathcal{S}$ and $\mathcal{H}$ are tangent along a circumference (see Picture TIc in Table 2) and the roots $\lambda_i$, $i=1,\dots,4$, satisfy 
\[-a^2=\lambda_1=\lambda_2= \lambda_3<0<\lambda_4.\]
\item $\mathcal{S}$ and $\mathcal{H}$ are not in contact (see Picture I in Table 1) and the roots $\lambda_i$, $i=1,\dots,4$, satisfy  
 \[-a^2=\lambda_1=\lambda_2< \lambda_3<0<\lambda_4.\]
\end{enumerate}
\end{lemma}
\proof
Assume $x_c=y_c=0$. By Lemma~\ref{lemma:roota} this is equivalent to the fact that $-a^2$ is a multiple root, so we can assume $\lambda_1=\lambda_2=-a^2$. There is an extra negative root $\lambda_3$ and a positive root $\lambda_4$.  We analyze the negative root $\lambda_3$ depending on the value of $z_c$. By Corollary~\ref{coro:rootatriple} the negative root is given by \[\lambda_3=\left(c^2-r^2+z_c^2-\sqrt{\left(c^2-r^2+z_c^2\right)^2+4 c^2 r^2}\right)/2.\]
Note that the derivative of $\lambda_3$ with respect to $z_c^2$ is positive so we conclude that $\lambda_3$ grows as $z_c^2$ increases. 
We analyze the cases $r<a$ and $a\leq r$ separately. If $r<a$, note from expression \eqref{eq:poly3} in Lemma~\ref{lemma:tecnical-results-1} that the roots of $f$ are $\lambda_1=\lambda_2=-a^2$, $\lambda_3=-r^2$ and $\lambda_4=c^2$ if the sphere is centered at $(0,0,0)$, so the result follows because $\lambda_3$ increases with $z_c^2$. Assume henceforth that $r\geq a$. 
By Lemma~\ref{lemma:rootatangent} tangency between $\mathcal{S}$ and $\mathcal{H}$ is characterized by $\lambda_3=-a^2=\lambda_1=\lambda_2$ and, thereby, we obtain the three given possibilities (note that we have renamed $\lambda_2$ and $\lambda_3$ in assertion (1) for coherence with the notation in Table 2). Since each of the three relative positions matches one of the three possible configuration of roots (see Lemma~\ref{lemma:distribucion-raices}), the conditions are necessary and sufficient.
\qed

\begin{lemma}\label{lemma:tangency-rootc} 
If $c^2$ is a multiple root, then $\mathcal{H}$ and $\mathcal{S}$ are tangent and $z_c=0$. Moreover, in this case,
\begin{enumerate}
\item $c^2$ is a double root if and only if $\mathcal{H}$ and $\mathcal{S}$ are tangent in two points of the same vertical ray (i.e. in two points of the form $(x,y,\pm z)$) and $c^2<ar$ (see Picture~TEs in Table 3). 

\item $c^2$ is a triple root if and only if $\mathcal{H}$ and $\mathcal{S}$ are tangent at one point of the plane $z=0$ and $ar=c^2$.  
\end{enumerate}
\end{lemma}
\proof
From Lemma~\ref{lemma:rootc} we know that $c^2$ is a multiple root if and only if $z_c=0$ and $r^2=-c^2+c^2( x_c^2 + y_c^2)/(a^2 + c^2)$. Now, we check that the latter condition corresponds to the existence of tangency. We study the solutions of the system given by the equations $X_0^tHX_0=0$ and $X_0^tSX_0=0$, which for $X_0^t=(x,y,z,1)$ in homogeneous coordinates correspond to the equations
\[
\frac{x^2}{a^2} + \frac{y^2}{a^2} - \frac{z^2}{c^2}=1 \text{ and } (x-x_c)^2+(y-y_c)^2+z^2=r^2\,.
\]
We are going to analyze the solution that has the direction given by $(x_c,y_c)$ at the $XY$-plane. In other words, we intersect by the vertical plane that contains the origin and the intersection points between the hyperboloid and the sphere in a hypothetical tangency situation. Hence by a change to cylindrical coordinates we consider the following equations instead:
\begin{equation}\label{eq:intersection}
\frac{\rho^2}{a^2} - \frac{z^2}{c^2}=1 \text{ and } (\rho-\rho_c)^2+z^2=r^2\,.
\end{equation}
Substitute $z^2$ to see that 
\begin{equation}\label{eq:intersection-substitutedz2}
(1+\frac{c^2}{a^2})\rho^2-2 \rho_c \rho+\rho_c^2-c^2-r^2=0\,.
\end{equation}
Now, observe that if the discriminant vanishes then there is a unique solution for $\rho$. The discriminant vanishes if $r^2=-c^2+c^2\rho_c^2/(a^2 + c^2)$, which is precisely the condition for $c^2$ to be a multiple root. 

Assuming  $r^2=-c^2+c^2\rho_c^2/(a^2 + c^2)$ the coordinate $\rho$ of the intersection points is given by $\rho=a^2\rho_c/(a^2+c^2)$. On the other hand, a normal vector to $\mathcal{H}$ has expression $(\frac{\rho}{a^2},\frac{-z}{c^2})$ and a normal vector to $\mathcal{S}$ has expression $(\rho-\rho_c,z)$. This two vectors have the same direction if and only if $\rho=a^2\rho_c/(a^2+c^2)$. Therefore the intersection points are indeed tangency points.

We continue the analysis of the multiplicity of the root. First, substitute in the hyperboloid equation the expression for $\rho$ to see that the $z$ coordinate of the intersection point satisfies  $z^2=c^2(\frac{a^2\rho_c^2}{(a^2+c^2)^2}-1)$. We have two possibilities:
\begin{itemize}
\item $z=0$: from the previous expresion this is equivalent to $a^2\rho_c^2=(a^2+c^2)^2$. But, by Lemma~\ref{lemma:rootc} this means that $c^2$ is a triple root and $ar=c^2$. In this case we have only one point of tangency. This shows assertion (2).
\item $z\neq 0$: so $a^2\rho_c^2\neq (a^2+c^2)^2$ and, by Lemma~\ref{lemma:rootc} $c^2$ is a double root. Moreover, in this case we have two points of tangency in the same vertical ray of the form $(x,y,\pm z)$. For this to be possible it is necessary that $c^2<ar$, as we saw in Remark~\ref{remark:relationarc2}. This shows assertion (1).\qed
\end{itemize}

Note that if the center of $\mathcal{S}$ is in the plane $z=0$ then $-a^2$ and $c^2$ are roots of the characteristic polynomial, as follows from Lemmas~\ref{lemma:tecnical-results-1} and \ref{lemma:rootc}. We consider a generic example with this particular property.
\begin{example}\label{example:example}
Consider generic $a$, $c$ and $r$. We build a generic example so that $\mathcal{S}$ is exterior to $\mathcal{H}$. Let the center of $\mathcal{S}$ be at $(x_c,0,0)$ where $x_c^2=(1+\frac{a^2}{c^2})(c^2+r^2)+(a+r)^2$. Thus, since $(1+\frac{a^2}{c^2})(c^2+r^2)<\rho_c^2$, we have from Equation~\ref{eq:intersection-substitutedz2} (proof of Lemma~\ref{lemma:tangency-rootc}) that there is not contact between $\mathcal{H}$ and $\mathcal{S}$. On the other hand, the roots of $f(\lambda)$ are $-a^2$, $c^2$ and $\lambda_\pm=\frac{1}{2} \left(-a^2-r^2+x_c^2\pm\sqrt{\left(a^2+r^2-x_c^2\right)^2-4 a^2 r^2}\right)$.
Now, substitute $x_c^2$ in the expression for $\lambda_\pm$ to see that  
\[
\lambda_\pm=\frac{1}{2} \left((1+\frac{a^2}{c^2})(c^2+r^2)+2ar\pm\sqrt{\left((1+\frac{a^2}{c^2})(c^2+r^2)+2ar\right)^2-4a^2r^2}\right).
\] 
Note that the radicand is positive, so $\lambda_\pm$ are two different real roots and, moreover, they are positive. Also, note from Lemma~\ref{lemma:rootc} that the condition for $c^2$ to be a multiple root is $r^2=-c^2+c^2 x_c^2/(a^2 + c^2)$, which is not satisfied for this value of $x_c^2$. Hence there are three positive distinct roots and one negative root for $f(\lambda)$ in this particular case.
\end{example}

Now we study different cases for $z_c=0$. Again we use cylindrical coordinates where $\rho_c=\sqrt{x_c^2+y_c^2}$. We distinguish when  $\rho_c=a+r$ and when  $\rho_c>a+r$. In the first case, a straightforward calculation gives the following lemma.
\begin{lemma}\label{lemma:TEs1}
Let the center of $\mathcal{S}$ be at $(\rho_c,\theta,0)$, for any $\theta\in[0,2\pi)$, with $\rho_c=a+r$. Then there is a tangent point at $(a,\theta,0)$ and some extra contact as in Picture TEs1 in Table 3 if $c^2<ar$. The roots of the characteristic polynomial in this case are $\lambda_1=-a^2$, $\lambda_2=c^2<\lambda_3=\lambda_4=ar$.
\end{lemma}

We study now the behaviour of the two unknown roots for the particular setting in which $\rho_c>a+r$. 

\begin{lemma}\label{lemma:planoxy}
Let the center of $\mathcal{S}$ be at $(\rho_c,\theta,0)$, for any $\theta\in[0,2\pi)$, with $\rho_c>a+r$. If $ar\leq c^2$ then the roots $\lambda_i$, $i=1,\dots,4$, satisfy 
$\lambda_1=-a^2<0<\lambda_2< ar<\lambda_3$, $\lambda_4=c^2$, whereas if $c^2<ar$, then one the following possibilities holds:
\begin{itemize}
\item $\mathcal{S}$ and $\mathcal{H}$ are in non-tangent contact (see Picture Cm in Table 3) and the roots $\lambda_i$, $i=1,\dots,4$, satisfy 
\[\lambda_1=-a^2<0<\lambda_2=c^2< \lambda_3<\lambda_4.\]
\item $\mathcal{S}$ and $\mathcal{H}$ are tangent (see Picture TEs in Table 3 and note that this tangency is realized in two points of a vertical ray) and the roots $\lambda_i$, $i=1,\dots,4$, satisfy 
\[\lambda_1=-a^2<0<\lambda_2=\lambda_3=c^2<\lambda_4.\]
\item $\mathcal{S}$ and $\mathcal{H}$ are not in contact (see Picture E in Table 1) and the roots $\lambda_i$, $i=1,\dots,4$, satisfy  
\[\lambda_1=-a^2<0<\lambda_2<\lambda_3=c^2<\lambda_4.\] 
\end{itemize}
\end{lemma}
\proof
As a consequence of Lemmas~\ref{lemma:tecnical-results-1} and \ref{lemma:rootc} and because $z_c=0$, we have that $-a^2,c^2$ are roots of the characteristic polynomial. Hence a direct computation shows that the other two roots are given by 
$$\lambda_{\pm}=\frac{1}{2}\left(\rho_c^2-a^2-r^2\pm\sqrt{(\rho_c^2-a^2-r^2)^2-4a^2r^2}\right).$$
First note that $\lambda_{\pm}$ are real and positive. Moreover, $\lambda_+=\lambda_-$ when $\rho_c=a+r$; in this case, $\lambda_{\pm}=ar$, so the eigenvalue structure satisfies:
$$\lambda_1=-a^2<0<\lambda_2=c^2,\,0< ar=\lambda_3=\lambda_4.$$

We turn our attention to the case $\rho_c>a+r$.
To study the behaviour of $\lambda_{\pm}$ as $\rho_c$ varies, we compute
 \[
 \frac{d}{d\rho_c^2}\lambda_\pm=\frac{1}{2}\left( 1\pm\frac{\rho_c^2-a^2-r^2}{\sqrt{\left(\rho_c^2-a^2-r^2\right)^2-4 a^2 r^2}}\right)
 \]
to see that $\dfrac{d}{d\rho_c^2}\lambda_+>0$ and $\dfrac{d}{d\rho_c^2}\lambda_-<0$ for $\rho_c>a+r$, since $\rho_c^2-a^2-r^2>0$ and $|\rho_c^2-a^2-r^2|>|\sqrt{\left(\rho_c^2-a^2-r^2\right)^2-4 a^2 r^2}|$. 
Hence, the eigenvalue $\lambda_+$ increases whereas $\lambda_-$ decreases as $\rho_c$ increases. As a consequence (using also Lemma~\ref{lemma:tangency-rootc}), if $ar\leq c^2$, then the eigenvalues configuration is 
\[
\lambda_1=-a^2<0<\lambda_2< ar<\lambda_3, \quad\lambda_4=c^2.
\]
However, if $c^2<ar$, using Lemma~\ref{lemma:tangency-rootc}, we have the following possibilites:
\begin{itemize}
	\item $\lambda_1=-a^2<0<\lambda_2=c^2<\lambda_3<ar<\lambda_4$ corresponding to non-tangent contact,
	\item $\lambda_1=-a^2<0<\lambda_2=c^2=\lambda_3<ar<\lambda_4$ corresponding to the tangent case, and
	\item $\lambda_1=-a^2<0<\lambda_2<\lambda_3=c^2<ar<\lambda_4$ corresponding to non contact.	
\end{itemize}
This completes the proof of the lemma.
\qed

\begin{lemma}\label{lemma:multipleroottangent}
Let $\lambda\neq -a^2$ be a real root of $f(\lambda)$. If the multiplicity of $\lambda$ is $m\geq 2$, then there exists at least one point where $\mathcal{H}$ and $\mathcal{S}$ are tangent.
\end{lemma}
\proof
If $\lambda=c^2$, then the result follows by Lemma~\ref{lemma:tangency-rootc}. Assume $\lambda\notin \{-a^2,c^2\}$ is a real root of $f(\lambda)$. We study the rank of $\lambda H+S$ triangularizing the matrix to obtain:
\[
\lambda H+S \simeq
\left(   \begin{array}{rrrc}
\frac{\lambda+a^2}{a^2}&0&0&-x_c\\
0&\frac{\lambda+a^2}{a^2}&0&-y_c\\
0&0&\frac{-\lambda+c^2}{c^2}&-z_c \\
0&0&0&-\lambda-r^2+\frac{\lambda}{\lambda+a^2}(x_c^2+y_c^2)+\frac{\lambda}{\lambda-c^2}z_c^2
\end{array} \right ).
\]
We have $\operatorname{rank}(\lambda H+S)=3$. Let $E^i=\operatorname{ker}(\lambda I+H^{-1}S)^i$ be the generalized eigenspace of order $i$. If $\lambda$ is a root of multiplicity greater than $1$, since $\operatorname{rank}(\lambda H+S)=3$, we have that $1=\operatorname{dim}(E^1)<\operatorname{dim}(E^2)$. Hence, there exists $Y\in E^2-E^1.$  Let $X_0=(\lambda I+H^{-1}S)Y=H^{-1}(\lambda H +S)Y\neq 0.$ 
  
Since  $(\lambda I+H^{-1}S)X_0=(\lambda I+H^{-1}S)^2Y=0$ we have that $X_0\in E^1$. Now we check that 
$X_0^tHX_0=0$ as follows:
\[
X_0^tHX_0=X_0^t (\lambda H + S) Y=
  [(\lambda H + S)X_0]^tY=[H(\lambda I + H^{-1}S)X_0]^tY=0,
\]
where we have used that $\lambda H + S$ is symmetric. Since $X_0^t(\lambda H+S)X_0=0$, it follows that $X_0^tSX_0=0$. Therefore $X_0$ is a point of the two quadrics. Also, since $(\lambda H+S)X_0=0$, it follows that $ -\lambda HX_0=SX_0$ so $H$ and $S$ are indeed tangent at $X_0$.
\qed

\begin{lemma}\label{lemma:tangent-implies-multipleroot}
If $\mathcal{H}$ and $\mathcal{S}$ are tangent, then there exists a multiple real root of the characteristic polynomial.
\end{lemma}
\proof
Assume $\mathcal{H}$ and $\mathcal{S}$ are tangent at $X_0$, then $X_0^tHX_0=X_0^tSX_0=0$ and $SX_0=-\alpha_0 HX_0$ for a certain $\alpha_0\in \mathbb{R}$. Hence $(\alpha_0H+S)X_0=0$, so $\alpha_0$ is a root of $f$. We argue by contradiction to show that there exists a multiple root of $f$. If there are no multiple roots, then there are $4$ simple real roots or $2$ simple real roots and $2$ complex conjugate roots. 

Assume there are $4$ different roots $\alpha_i$, either all of them real or two of them complex conjugate, with eigenvectors $X_i$ for $i=0,1,2,3$. The vectors $X_0,X_1,X_2,X_3$ are linearly independent. Since $X_i$ are eigenvectors, we have $X_i^t(\alpha_iH+S)X_0=0$ for all $i=0,1,2,3$. Hence, for $i=1,2,3,$ we have the relations
\[
\begin{array}{rcl}
\alpha_0 X_i^tHX_0+X_i^t S X_0=0,\\
\noalign{\medskip}
\alpha_i X_i^tHX_0+X_i^t S X_0=0.
\end{array}
\]
Substract to get that $(\alpha_0-\alpha_i)X_i^tHX_0=0$. Since $\alpha_0\neq\alpha_i$ we have that $X_i^tHX_0=0$ and, hence, $X_i^t S X_0=0$ too. So $X_i^tHX_0=0$ for $i=0,1,2,3$. If $X_i$ are all real eigenvectors this implies that $HX_0=0$ as $X_0,X_1,X_2,X_3$ are linearly independent. This is a contradiction as $H$ is regular and $X_0\neq0$. If two of the eigenvalues are complex conjugate, suppose without loss of generality that $\alpha_3=\bar\alpha_2$, so $X_3=\bar X_2$. Consider the set of real vectors $\{X_0,X_1,X_2+X_3,iX_2-iX_3\}$ and observe that $\operatorname{span}\{X_0,X_1,X_2+X_3,iX_2-iX_3\}=\operatorname{span}\{X_0,X_1,X_2,X_3\}$. Now we obtain the same contradiction as before, i.e. $HX_0=0$, that shows this case is not possible either. Therefore the result follows.
\qed

As a consequence of the previous lemmas, we have that the tangency condition is identified in terms of multiple roots. We state the following theorem.

\begin{theorem}\rm\label{th:non-tangency}
$\mathcal{H}$ and $\mathcal{S}$ are tangent if and only if one of the following possibilities holds:
\begin{enumerate}
\item there exists one multiple root $\lambda\neq -a^2$,
\item $-a^2$ is a triple root.
\end{enumerate}
\end{theorem}
\proof
Lemmas~\ref{lemma:multipleroottangent} and \ref{lemma:rootatangent} imply that if we have a multiple root $\lambda\neq -a^2$ or $-a^2$ is a triple root, then there is at least one tangent point between $\mathcal{S}$ and $\mathcal{H}$. Conversely, Lemma~\ref{lemma:tangent-implies-multipleroot} shows that if there is a tangent point, then there exist a multiple root. Moreover, this multiple root is a triple root if it is $-a^2$, as we saw in Lemmas~\ref{lemma:roota} and \ref{lemma:eixoz}. 
\qed

\subsection{Moving sphere}\label{subsect:moving-sphere}
\mbox{}

In this section we use a technique based on a moving sphere. Let us move the sphere $\mathcal{S}$ translating its center along a parametrized curve $\alpha(t)$, with $t\in [a,b]$, and denote it by $S(t)$. Define $f_t(\lambda)=\operatorname{det}(\lambda H+S(t))$ as the characteristic polynomial $f(\lambda)$ for $\mathcal{S}$ centered at $\alpha(t)$.

\begin{lemma}\label{lemma:complexroots}
\begin{enumerate}
\item Let $p(t)=\alpha_2(t) \lambda^2+\alpha_1(t) \lambda+\alpha_0(t)$ a polynomial of degree $2$ whose coefficients are continuous functions of $t\in [t_0,t_1]$. If $p(t_0)$ has two distinct real roots and $p(t_1)$ has complex conjugate roots, then there exists $t_d\in (t_0,t_1)$ such that $p(t_d)$ has a double root.
\item Let $S(t)$ be a moving sphere for $t\in [t_0,t_1]$ and $\mathcal{H}$ a fixed hyperboloid. If  $f_{t_0}(\lambda)$ has four real roots and $f_{t_1}(\lambda)$ has two real  and two complex roots, then there exists $t_d\in (t_0,t_1)$ such that $f_{t_d}(\lambda)$ has a multiple root.
\end{enumerate}
\end{lemma}
\proof Since $\alpha_i$ are continuous functions for $i=0,1,2$, so is the discriminant $d(t)=\alpha_1(t)^2-4\alpha_0(t)\alpha_2(t)$. Since $d(t_0)>0$ and $d(t_1)<0$, the result follows as a consequence of Bolzano's Theorem.

Note first that $-a^2$ is a root. Let $\lambda_2(t)$, $\lambda_3(t)$ and $\lambda_4(t)$ be the other three roots. For every $t\in [t_0,t_1]$, at least one of $\lambda_2(t)$, $\lambda_3(t)$, $\lambda_4(t)$ is real.  Assume without loss of generality that $\lambda_2(t)$ is real for $[t_0,t_1^\prime]$ with $t_1^\prime$ maximal so that $\lambda_3(t_1^\prime)$, $\lambda_4(t_1^\prime)$ are complex conjugate. 
We consider the polynomial $\tilde f(t)=f_t(\lambda)/((x+a^2)(x-\lambda_2(t)))$ defined for  $t\in[t_0,t_1^\prime]$. $\tilde f$ is a polynomial of degree $2$ whose coefficients are continuous functions of $t$, since the roots of a polynomial with continuous functions coefficients are continuous.  
Hence we apply (1) to $\tilde f$ in $[t_0,t_1^\prime]$ to obtain (2).
\qed

\subsection{Characterization of relative positions}\label{subsect:characterization-relative-positions}

In this section we use the previous lemma to prove the result that  characterizes the different relative positions between $\mathcal{S}$ and $\mathcal{H}$. We begin by giving some examples which will be useful in the proofs.

\begin{example}\label{example:complex-roots}
There are two examples of special significance for the subsequent analysis of complex roots.
\begin{enumerate}
\item Let $\mathcal{S}$ be the sphere with radious $r$, such that $r<2a$, and center at $(a,0,0)$. It is easy to check that there is contact between $\mathcal{S}$ and $\mathcal{H}$. We study the roots of $f(\lambda)$ to see that they are $-a^2$, $c^2$ and $\lambda_\pm=\frac{1}{2} \left(-r^2\pm r\sqrt{r^2-4 a^2}\right)$. So we have that $\lambda_\pm$ are complex roots (because $r<2a$).
\item Let $\mathcal{S}$ be the sphere with radious $r$, such that $a<2r$, and center at $(r,0,0)$. Again, there is contact between $\mathcal{S}$ and $\mathcal{H}$. The roots of $f(\lambda)$ are $-a^2$, $c^2$ and $\lambda_\pm=\frac12\left(-a^2\pm a \sqrt{a^2-4  r^2}\right)$. Hence, as $a<2r$ we have that $\lambda_\pm$ are complex conjugate roots.
\end{enumerate}
\end{example}

We characterize the non-contact possible situations between $\mathcal{S}$ and $\mathcal{H}$ in the following lemma.

\begin{lemma}\label{lemma:nocontact}
Assume there is no contact between $\mathcal{S}$ and $\mathcal{H}$. Then one of the following holds:
\begin{enumerate}
\item If $\mathcal{S}$ is interior to $\mathcal{H}$, then there is a positive root, $-a^2$ is a root and there are two negative distinct roots $\lambda_2$, $\lambda_3$ such that $-a^2\leq \lambda_2<\lambda_3$.
\item If $\mathcal{S}$ is exterior to $\mathcal{H}$, then there are 3 positive distinct roots. Moreover, these roots $\lambda_2$, $\lambda_3$ and $\lambda_4$ satisfy $0<\lambda_2<\lambda_3\leq c^2 <\lambda_4$
or $0<\lambda_2<\lambda_3< c^2\leq\lambda_4$.
\end{enumerate}
\end{lemma}
\proof
Assume $\mathcal{S}$ is interior to $\mathcal{H}$ and there is no contact. If $x_c=y_c=0$, the center of $\mathcal{S}$ is in the $OZ$ axis and that case was previously considered in Lemma~\ref{lemma:eixoz}, from where it follows the result for that particular situation. Assume $x_c\neq 0$ or $y_c\neq 0$ henceforth. By Lemma~\ref{lemma:roota} we have that $-a^2$ is not a double root. The interior of $\mathcal{H}$ for a fixed value of $z$ is a disk, so it is convex. Hence we can build a path $\alpha(t)=((1-t) x_c,(1-t) y_c,z_c)$, $t\in [0,1]$, so that $\alpha(0)$ is the center of $\mathcal{S}$, $\alpha(1)$ is in the OZ axis, and we move the center of the sphere along $\alpha$ without touching $\mathcal{H}$. As before, define $f_t(\lambda)$ to be the characteristic polynomial when the center of $\mathcal{S}$ is at $\alpha(t)$. From Lemma~\ref{lemma:eixoz}  we know that  $f_1(\lambda)$ has roots $\lambda_i(1)$, $i=1,\dots,4,$ satisfying $-a^2=\lambda_1(1)=\lambda_2(1)<\lambda_3(1)<0<\lambda_4(1)$. Since the coefficients of $f_t(\lambda)$ are continuous functions on the variable $t$, so are the roots $\lambda_i(t)$ of $f_t(\lambda)$. Since there is no contact between $\mathcal{S}$ and $\mathcal{H}$  when the center of $\mathcal{S}$ moves along $\alpha$, we have that:
\begin{itemize}
\item there are no complex roots for $t\in [0,1]$: if there were a complex root, by Lemma~\ref{lemma:complexroots}, there would exist $t_d\in (0,1)$ so that there is a double root different from $-a^2$ for $f_{t_d}(\lambda)$. But, by Lemma~\ref{lemma:multipleroottangent}, that would imply that there is a tangent point, which is not the case.
\item there are 3 negative roots for $t\in [0,1]$: since there are 3 negative roots for $t=1$ and the roots are continuous functions which never take the value $0$ or complex values, there are always $3$ negative roots.
\item at least one root belongs to the interval $(-a^2,0)$: we have that $\lambda_3(1)\in (-a^2,0)$, that $\lambda_3$ is a continuous function,  that $-a^2$ is not a double root for $t\in[0,1)$ (by Lemma~\ref{lemma:roota}), that $0$ is not a root and that there are no complex roots along $\alpha$. Then $\lambda_3(t)\in  (-a^2,0)$ for $t\in [0,1]$. 
\end{itemize}
Now, by Lemma~\ref{lemma:distribucion-raices}, it follows that $\lambda_2(1)\in[-a^2,\lambda_3(1)]$. Since $\lambda_2(1)=\lambda_3(1)$ would imply tangency by Lemma~\ref{lemma:multipleroottangent}, we conclude $\lambda_2(1)\in[-a^2,\lambda_3(1))$ and assertion (1) follows.

To prove assertion (2), we use a similar strategy. Assume $\mathcal{S}$ is exterior to $\mathcal{H}$ and there is no contact. We build a path from the center of $\mathcal{S}$ to a point in the $XY$-plane. This can be done without contact between the sphere and the hyperboloid in the following way. Due to the connectivity of every vertical segment joining two points of $\mathcal{H}$, we can build a path $\beta(t)= (x_c,y_c, (1-t)z_c)$, $t\in [0,1]$ so that $\beta(0)$ is the center of $\mathcal{S}$ and $\beta(1)$ is the projection in the $XY$-plane. Hence $\beta$ is a path joining $(x_c,y_c,z_c)$ and $(x_c ,y_c,0)$ without contact between the sphere and the hyperboloid when the center of $\mathcal{S}$ moves along the path. Since there is no contact between $\mathcal{H}$ and $\mathcal{S}$, we have that $\rho_c>a+r$, so we use Lemma~\ref{lemma:planoxy} to see that the roots of $f_1(\lambda)$ satisfy  $\lambda_1=-a^2<0<\lambda_2< ar<\lambda_3$ and $\lambda_4=c^2$, if $ar\leq c^2$, or $\lambda_1=-a^2<0<\lambda_2<\lambda_3=c^2<ar<\lambda_4$ if $c^2<ar$. We use the previous lemmas to obtain assertion (2) as follows:
\begin{itemize}
\item there are no complex roots for $t\in [0,1]$: by Lemmas~\ref{lemma:complexroots} and \ref{lemma:multipleroottangent} that would imply that there is a tangent point, which is not the case.
\item there are three positive roots for $t\in [0,1]$: since there are three positive roots at the final point of the path and the roots are continuous functions which never take the value $0$ or complex values, there are always three positive roots.
\item  the three positive roots are different: as a consequence of Lemma~\ref{lemma:multipleroottangent} a multiple root would imply tangency, which is not the case as we have seen.
\end{itemize}
Now, by Lemma~\ref{lemma:distribucion-raices} we know that it is not possible that $\lambda_2<c^2<\lambda_3<\lambda_4$ so necessarily  $0<\lambda_2<\lambda_3< c^2 <\lambda_4$
if $z_c\neq 0$. This completes the proof of the lemma.
\qed 

\medskip

Henceforth we are going to consider non-tangent contact. If $r<a$, the intersection between $\mathcal{S}$ and $\mathcal{H}$ is a curve with only one connected component. However, if $r\geq a$ we distinguish two kinds of possible non-tangent contact: the intersection between $\mathcal{S}$ and $\mathcal{H}$ is a curve with one or with two connected components. 
First, we identify the meaning of complex roots in terms of the relative position between $\mathcal{S}$ and $\mathcal{H}$ for the case in which the intersection curve has one connected component.

\begin{lemma}\label{lemma:contact-implies-complexroots}
If $\mathcal{S}$ and $\mathcal{H}$ intersect only in a connected curve (without tangency) then there are two complex conjugate roots.
\end{lemma}
\begin{proof}
We distinguish two cases: $r\leq a$ and $r>a$. Assume first $r\leq a$. We are going to build a continuous path from $(x_c,y_c,z_c)$ to $(a,0,0)$ and use Example~\ref{example:complex-roots}(1). 

For cylindrical coordinates $\rho_c=\sqrt{x_c^2+y_c^2}$ and $\theta_c=\arctan \frac{y}x$, we consider the continuous path in cartesian coordinates from $(x_c,y_c,z_c)$ to $(a,0,0)$ given by
$$\gamma(t)=\left\{\begin{array}{lcr}
\left(\rho_c\cos((1-4t)\theta_0),\rho_c\sin((1-4t)\theta_0),z_c \right)
&  & \mbox{ if } 0\leq t\leq\frac{1}{4}, \\
& &\\
\big((2-4t)\rho_c+(4t-1)\frac{a}{c}\sqrt{c^2+z_c^2},0,z_c\big) &  & \mbox{ if } \frac{1}{4}\leq t\leq\frac{1}{2}, \\
& & \\
\big(\frac{a}{c}\sqrt{c^2+(2-2t)^2z_c^2},0,(2-2t)z_c\big) &  & \mbox{ if } \frac{1}{2}\leq t\leq 1,
\end{array}\right.$$

\noindent The points  $\gamma(0)$, $\gamma(1/4)$, $\gamma(1/2)$ and $\gamma(1)$ correspond, respectively, to points $C$, $D$, $E$ and $\tilde C$ in the picture  below


\begin{center}
\includegraphics[scale=0.25]{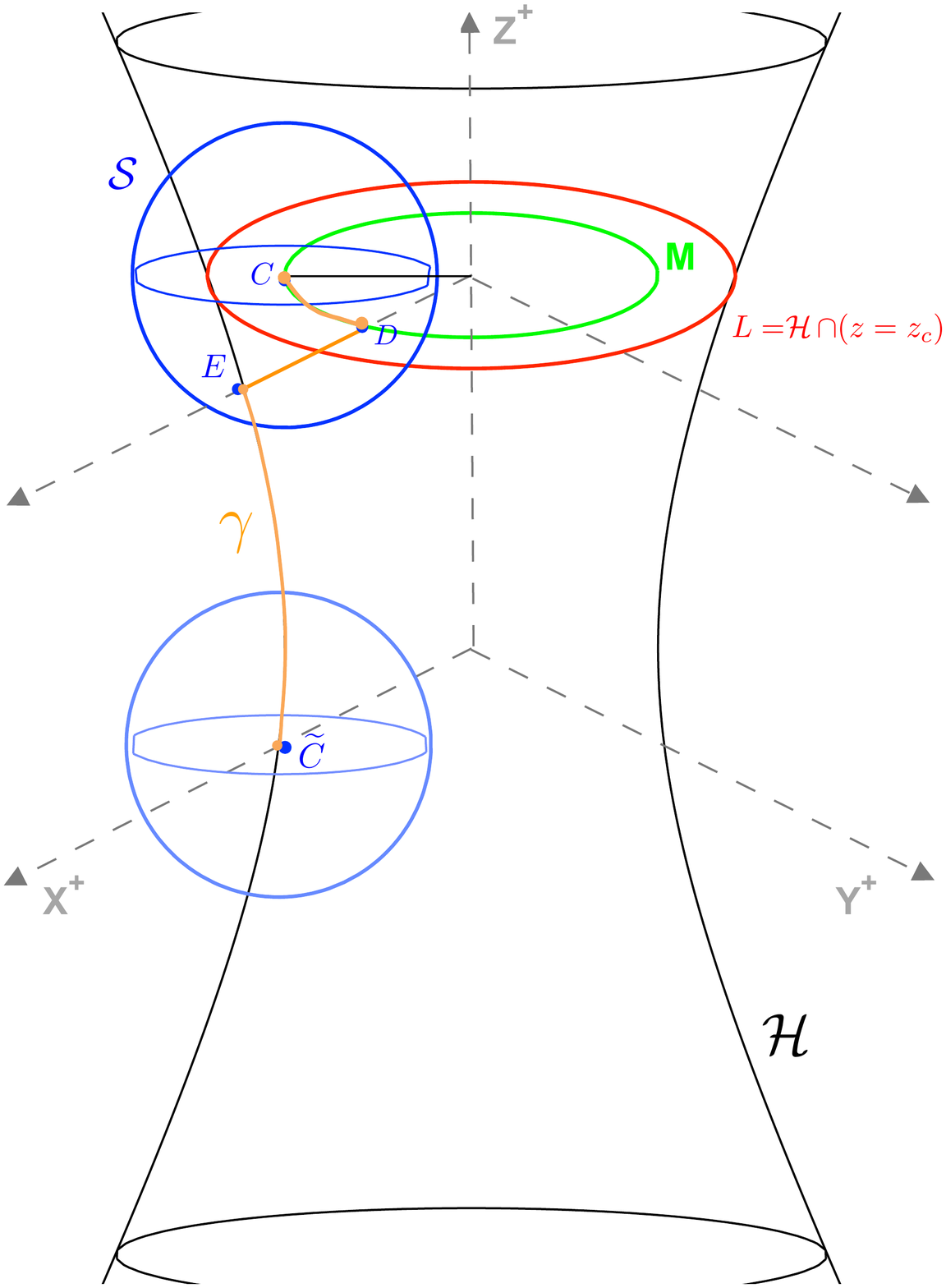}
\end{center}

 \noindent As Example~\ref{example:complex-roots} shows, there are two complex conjugate roots at the end of the path. Moreover, all along the path there is non-tangent contact between $\mathcal{S}$ and $\mathcal{H}$. As a consequence of Theorem~\ref{th:non-tangency}, there are no multiple roots all along the path. This means, by Lemma~\ref{lemma:complexroots}, that there are two complex roots for all $t\in[0,1]$ and, in particular, at the initial point. This completes the proof for $r< 2a$.   

Now assume $r\geq a$. The argument is analogous to the previous one, but we use Example~\ref{example:complex-roots}(2) and now the path is given by
$$\beta(t)=\left\{\begin{array}{lcr}
\left(\rho_c\cos((1-4t)\theta_0),\rho_c\sin((1-4t)\theta_0),z_c \right)
&  & \mbox{ if } 0\leq t\leq\frac{1}{4}, \\
&  &\\
\big((2-4t)\rho_c+(4t-1)\left[r-a+\frac{a}{c}\sqrt{c^2+z_c^2}\right],0,z_c\big) &  & \mbox{ if } \frac{1}{4}\leq t\leq\frac{1}{2}, \\
& & \\
\big(r-a+\frac{a}{c}\sqrt{c^2+(2-2t)^2z_c^2},0,(2-2t)z_c\big) &  & \mbox{ para } \frac{1}{2}\leq t\leq 1.\end{array}\right.$$
\noindent Thus one has that $\beta(0)=(x_c,y_c,z_c)$, that $\beta(1/4)=(\rho_c,0,z_c)$, that   $\beta(1/2)=(r-a+\frac{a}{c}\sqrt{c^2+z_c^2},0,z_c)$ and that $\beta(1)=(r,0,0)$. Again, there is non tangent contact all along the path and at the final point we have complex roots, so the previous argument also applies in this case. This completes the proof of the lemma.
\end{proof}

Lemma~\ref{lemma:contact-implies-complexroots} identifies  when the intersection is connected in terms of the roots.
Now we study the case in which the intersection curve between $\mathcal{S}$ and $\mathcal{H}$ has two connected components. There are two situations that are different in essence and are identified by a different configuration of roots: one which corresponds to Picture Ca in Table 2, and other that corresponds to Picture Cm in Table 3. In the first case we need that $a<r$ and in the second one that $c^2<ar$. See Lemmas~\ref{lemma:eixoz}(1) and \ref{lemma:planoxy}(1) for examples of these two situations.

\begin{lemma}\label{lemma:contact-2connectedcomponents}
If $\mathcal{S}$ and $\mathcal{H}$ intersect in a curve with two connected components then 
\begin{enumerate}
\item for Type~Ca in Table 2 the roots are $\lambda_1=-a^2$, $\lambda_2$, $\lambda_3$ and $\lambda_4$ such that $\lambda_2<\lambda_3\leq-a^2$, and $\lambda_4>0$;
\item for Type~Cm in Table 3 the roots are $\lambda_1=-a^2$, $\lambda_2$, $\lambda_3$ and $\lambda_4$ such that  $0<c^2\leq\lambda_2<\lambda_3<\lambda_4$.
\end{enumerate}
\end{lemma}
\begin{proof}
The argument  of the proof is similar to previous ones. In the first case, corresponding to the relative position represented by Type~Ca (Table 2), we are going to build a path for the center of the sphere from $(x_c,y_c,z_c)$ to $(0,0,z_c)$. Consider the parametrized curve $\alpha(t)=(x_c (1-t),y_c (1-t), z_c)$ for $t\in[0,1]$. Let $\lambda_1(t)$, $\lambda_2(t)$, $\lambda_3(t)$ and $\lambda_4(t)$ be the roots of $f_t(\lambda)$. In virtue of Lemma~\ref{lemma:eixoz}, the roots at $\alpha(1)$ satisfy $\lambda_2(1)<-a^2=\lambda_1(1)=\lambda_3(1)< 0<\lambda_4(1)$. Note that $\mathcal{S}$ and $\mathcal{H}$ are not tangent for any value of $t\in [0,1]$, so $\lambda_2\neq -a^2$ (Lemma~\ref{lemma:rootatangent}) and, as a consequence of  Lemma~\ref{lemma:complexroots}, $\lambda_2(t)$ cannot be complex, so $\lambda_2(t)<-a^2$ for $t\in [0,1]$. Now, by Lemma~\ref{lemma:distribucion-raices}, the configuration of roots follows. 

Consider now the case represented in Picture Cm (Table 3). Acting in the same fashion we build a continuous path from the center of $\mathcal{S}$ to the $XY$-plane as $\beta(t)=(x_c,y_c,(1-t)z_c)$. Since $\beta(1)$ belongs to the $XY$-plane, Lemma~\ref{lemma:planoxy} applies and we have roots $\lambda_i(t)$ satisfying $\lambda_1(1)=-a^2<0<c^2= \lambda_2(1)< \lambda_3(1)<\lambda_4(1)$. By continuity of the path and roots, and using Lemma~\ref{lemma:distribucion-raices} to see that $c^2\leq \lambda_2(t)$, we get that
$\lambda_1(t)=-a^2<0<c^2\leq \lambda_2(t)< \lambda_3(t)<\lambda_4(t)$ for all $t\in[0,1]$. 
\end{proof}

\section{Proofs of the main results}\label{sect:proofs-ths}

In this section we prove the results stated in Section~\ref{sect:main-results} as a consequence of the analysis developed in Section~\ref{sect:analysis}. 

For the proofs of Theorems~\ref{th:th1}, \ref{th:th2} and \ref{th:th3} we proceed in the following way: we will show that for each relative position between $\mathcal{H}$ and $\mathcal{S}$, the roots of the characteristic polynomial satisfy the relations given in the corresponding table. This is, we show that this relations are necessary conditions. Since all the relative positions are mutually exclusive (and so are the configurations of roots given in Table~1, Table~2 and Table~3) and they cover all relative positions between $\mathcal{H}$ and $\mathcal{S}$, it automatically follows that the conditions are also sufficient. Hence, a certain configuration of roots implies the corresponding relative position given in the tables. 

{\it Proof of Theorem~\ref{th:th1}.} We establish one implication, showing that for a given relative position between $\mathcal{H}$ and $\mathcal{S}$ the roots of the characteristic polynomial satisfy the relations given in Table 1. 

\begin{enumerate}
\item Types~I and E: these follow directly from Lemma~\ref{lemma:nocontact}.
\item Types~TI and TE: by Lemma~\ref{lemma:tangent-implies-multipleroot} we have that there exists a multiple root. Moreover, these two relative positions can be obtained by moving a sphere continously from positions of Type~I and E, respectively. Since the roots are also continuous, the two configurations of roots are obtained as a consequence of this fact.
\item Type~C: this follows directly from Lemma~\ref{lemma:contact-implies-complexroots}.\qed
\end{enumerate}

\medskip

{\it Proof of Corollary}~\ref{co:rlessa}.
Under the hypothesis $r<a$ and $ar<c^2$, all possible relative positions are given in Theorem~\ref{th:th1}. In particular, if there is contact between $\mathcal{H}$ and $\mathcal{S}$, the intersection curve is connected.
 
Assume there is contact between $\mathcal{H}$ and $\mathcal{S}$. By Lemma~\ref{lemma:contact-implies-complexroots}, if the contact is non-tangent then there are complex roots, whereas Theorem~\ref{th:non-tangency} shows that for tangent contact there is a double root different from $-a^2$ or $-a^2$ is a triple root. The latter case is not possible since, by Lemma~\ref{lemma:roota}, if $-a^2$ is a triple root then $a\leq r$, which contradicts one of the hypothesis.
 
Reciprocally, if there are complex roots, then we use  Lemma~\ref{lemma:nocontact} to conclude that there is contact. If there is a double root different from $-a^2$ then, by Theorem~\ref{th:non-tangency}, there is tangent contact. \qed

\medskip

{\it Proof of Theorem~\ref{th:th2}.} As before, we show that the relative positions described in Table~2 imply the corresponding configuration of roots. Assume $a\leq r$ henceforth. 
\begin{enumerate}
\item Type TIc: it follows from Lemma~\ref{lemma:rootatangent}.
\item Type Ca: it was treated in Lemma~\ref{lemma:contact-2connectedcomponents}(1).
\item Type Td: Lemma~\ref{lemma:tangent-implies-multipleroot} shows there is a multiple root. Now, since we can obtain a Type Td relative position by moving a sphere along a continuous path from a Type Ca position, the configuration follows.\qed
\end{enumerate}

\medskip

{\it Proof of Theorem~\ref{th:th3}.}  Assume $c^2<ar$. Again we argue to proof the necessity of the root configurations:
\begin{enumerate}
\item Type TEs: it follows from Lemma~\ref{lemma:rootc} and Lemma~\ref{lemma:tangency-rootc} (see also Lemma~\ref{lemma:planoxy}).
\item Type TEs1: this corresponds to Lemma~\ref{lemma:TEs1}.
\item Type Cm: it corresponds to Lemma~\ref{lemma:contact-2connectedcomponents}(2).
\item Type TEs2: by Lemma~\ref{lemma:tangent-implies-multipleroot} there is a multiple root. As one can move the sphere along a continuous path from Type Cm to Type TEs2, the roots configuration follows. \qed
\end{enumerate}

\section*{Acknowledgements}
This work was done by the members of the Research Group {\it Xeometría Diferencial e as súas Aplicacións} (Universidade da Coruña), which was founded by Professor Eugenio Merino Gayoso. The authors would like to dedicate this work to the memory of his friend and colleague Eugenio, who passed away in December, 2012.


\begin{thebibliography}{99}
\bibitem{Birkhoff}
Birkhoff, G., MacLane, S., {\em A Survey of Modern Algebra}, New York, Macmillan, 1996.
	


\bibitem{docarmo}
do Carmo, M. P., {\em Differential Geometry of Curves and Surfaces}, Prentice-Hall, Inc., 1976.



\bibitem{Choi2003} 
Choi, Y.-K., Wang, W., Kim, M.S.,  Exact collision detection of two moving ellipsoids under rational motions,  Proc. IEEE Conf. Robot. Autom., pp.  349--354. Taipei, Taiwan, 2003.

\bibitem{Choi2006}
Choi, Y.-K., Wang, w., Liu Y., Kim, M-S, Continuous Collision Detection for Two Moving Elliptic Disks, IEEE Transactions on Robotics {\bf 22}, 2, (2006), 213--224. 


\bibitem{etayo2006}
 Etayo, F., Gonzalez-Vega, L., del Rio, N., A new approach to characterizing the relative position of two ellipses depending on one parameter, {\it Computer Aided Geometric Design}
{\bf 23} (4), (2006) 324--350. 

\bibitem{Farouki1989}
 R. T. Farouki, C. Neff,  	M. A. O'Conner, Automatic parsing of degenerate quadric-surface intersections, ACM Transactions on Graphics {\bf 8} 3, (1989), 174-203. 


\bibitem{44}
Fischer, K., Gartner, B., The Smallest Enclosing Ball of Balls: Combinatorial Structure and Algorithms, Proceedings of 19th Annual Symposium on Computational Geometry (SCG), pp. 291--301, 2003.

\bibitem{39}
Hubbard, P. M., Approximating Polyhedra with Spheres for Time-critical Collision Detection, ACM Transactions on Graphics, vol. 15, no. 3, 179--210, 1996.

\bibitem{X2011} 
Jia, X., Choi, Yi-K.,  Mourrain, B.,  Wang, W., An Algebraic Approach to Continuous Collision Detection for Ellipsoids. Computer Aided Geometric Design, Elsevier, 28 (3), 164--176, 2011.

   

\bibitem{levin76}
J. Levin,
A parametric algorithm for drawing pictures of solid objects composed of quadric surfaces,
{\it Comm. ACM} {\bf 19} (10) 1976, 555--563.

\bibitem{levin78}
J. Levin,
Mathematical models for determining the intersections of quadric surfaces, 
{\it Computer Graphics and Image Processing}, {\bf 11} (1) (1979), 73--87.



\bibitem{9} 
Lin, M. C., Gottschalk, S., Collision detection between geometric models: a survey.   Proceedings of IMA Conference on Mathematics of Surfaces, 1998.

\bibitem{Lingeo} 
Lin, X.,  Ng, Tang.-Tat,  Contact Detection Algorithms for Three-Dimensional Ellipsoids in Discrete Element Modelling, International Journal for Numerical and Analytical Methods in Geomechanics, 
19(9), 653--659, 1995.

\bibitem{liu2004}
Liu, Y., Chen, F.-L., Algebraic conditions for classifying the positional relationships between two conics and their applications,  {\it Journal of Computer Science and Technology}, {\bf 19} (5), (2004) 665--673. \black

\bibitem{ambrosio}  
Lopes, D.S., Silva, M.T., Ambrosio, J.A., Flores, P., A  mathematical framework for contact detection between quadric and superquadric surfaces. Multibody System Dynamics, 24 (3), 255--280, 2010.

\bibitem{34} 
O' Sullivan, C., Dinglina, J.,  Real-time Collision Detection and Response using Sphere-trees, Spring Conference on Computer Graphics,  pp.  83--92. Bratislava, Slovakia, 1999.




\bibitem{Wang 2004} 
Wang, W., Choi, Y.K., Chan, B., Kim, M.S., Wang, J., Efficient collision detection for moving ellipsoids using separating planes. Computing {\bf 72}, 235--246, 2004.

\bibitem{wank-wang-kim}
 Wang, W., Wang, J., Kim, M.-S.; 
An algebraic condition for the separation of two ellipsoids, 
Computer Aided Geometric Design, {\bf 18} (6) (2001), pp. 531--539.

\bibitem{42}
Welzl, E., Smallest Enclosing Disks (balls and ellipsoids),   
  New Results and New Trends in Computer Science, Lecture Notes in Computer Science 555, (1991), pp. 359--370.

\bibitem{MW}
Wonenburger, M.; Simultaneous Diagonalization of Symmetric Bilinear Forms, Journal of Mathematics and Mechanics, {\bf 15} (4) (1966), pp. 617--622.

\end{thebibliography}
\end{document}